\pdfoutput=1
\documentclass[12pt]{amsart}
\usepackage[margin=3.31cm]{geometry}
\renewcommand{\arraystretch}{1.1}

\usepackage{hyperref}
  \hypersetup{
   colorlinks,
		linkcolor={blue},
    citecolor={blue},
    urlcolor={black}
		}

\usepackage{enumitem}
\usepackage{graphicx,color}
\usepackage{wrapfig}
\usepackage[cp1250]{inputenc}
\usepackage{lpic}
\usepackage{amssymb, amsmath}
\usepackage[sc]{mathpazo}
\usepackage{caption} 
\usepackage{longtable}

\usepackage{import}

\theoremstyle{plain}
\newtheorem{theorem}{Theorem}
\newtheorem{proposition}[theorem]{Proposition}
\newtheorem{lemma}[theorem]{Lemma}
\newtheorem{corollary}[theorem]{Corollary}

\theoremstyle{definition}

\newtheorem{remark}[theorem]{Remark}
\newtheorem{question}[theorem]{Question}

\allowdisplaybreaks

\begin{document}

\title[Triple-crossing projections, moves and minimal diagrams]{Triple-crossing projections, moves on knots\\ and links, and their minimal diagrams}

\author[M. Jab{\l}onowski]{Micha{\l} Jab{\l}onowski}

\address{\parbox{\linewidth}{Micha{\l} Jab{\l}onowski, Institute of Mathematics, Faculty of Mathematics, Physics\\ and Informatics, University of Gda\'nsk, 80-308 Gda\'nsk, Poland\\\ }}

\email[Corresponding author]{michal.jablonowski@gmail.com}
\urladdr{\url{https://mat.ug.edu.pl/~mjablon}}

\author[{\L}. Trojanowski]{{\L}ukasz Trojanowski}

\address{{\L}ukasz Trojanowski}

\email{hehosek@gmail.com}



\date{\today}

\begin{abstract}
In this paper we present a systematic method to generate prime knot and prime link minimal triple-point projections, and then classify all classical prime knots and prime links with triple-crossing number at most four. We also extend the table of known knots and links with triple-crossing number equal to five. By introducing a new type of diagrammatic move, we reduce the number of generating moves on triple-crossing diagrams, and derive a minimal generating set of moves connecting triple-crossing diagrams of the same knot.
\end{abstract}

\maketitle

\section{Introduction}

It is known that any knot or link has a diagram where, at each its multiple point in the plane, only three strands are allowed to cross (pairwise transversely), see \cite{Ada13}. Such triple-point diagrams has been studied in several recent papers, such as \cite{Ada13, AHP19, ACFIPVWZ14, Nis19}.
\par
A diagrammatic \emph{move} is a transformation which can be used to alter a diagram while not changing the underlying knot or link type. In \cite{AHP19} it is shown that a set of five types of moves on triple-point diagrams ($1$-moves, $2$-moves, basepoint moves, band moves, and trivial pass moves) are sufficient to pass between any two diagrams of the same link up to certain equivalence. In this paper, we reduce the number of moves (e.g. the basepoint move follows from the other moves). By introducing $J$-type moves (that are similar to band moves), we prove that the set of all five mentioned moves is generated by two types of moves (the $J$-move and the trivial pass move). 
\par
Later in this paper, we systematically generate prime knot and prime link minimal projections, and then classify all prime knots and prime links with triple-crossing number at most four. We also extend the table of known knots and links with triple-crossing number equal to five.
\par
The paper is organized as follows. In Section\;\ref{s2}, we review the triple-crossing projections and triple-crossing diagrams. In Section\;\ref{s6}, we introduce a type $J$-move and used it to reduce the number of generating moves on triple-crossing diagrams and derive a minimal generating set of moves connecting triple-crossing diagrams of the same knot. In Section\;\ref{s3}, we present a method to generate reduced prime triple-crossing knot and link projections, together with their enumeration and a graphical presentation. We also relate them to spherical quadrangulations and hexagonations. In Section\;\ref{s4}, we identify, by using classical invariants, all prime knots and links on these projections, and generate their minimal diagrams. Our algorithms were implemented in \texttt{Wolfram Language 11}. We use the knot and link names used in the \cite{katlas} database, and for $12$-crossing links we use the Hoste-Thistlethwaite database.

\section{Definitions}\label{s2}

The \emph{projection} of a knot or link $KL\subset \mathbb{R}^3$ is its image under the standard orthogonal projection $\pi:\mathbb{R}^3\to\mathbb{R}^2$ (or into a $2$-sphere) such that it has only a finite number of self-intersections, called \emph{multiple points}, and in each multiple point each pair of its strands are transverse. If each multiple point of a projection has multiplicity three then we call this projection a \emph{triple-crossing projection}.
\par
The \emph{triple-crossing} is a three-strand crossing with the strand labeled $T, M, B$, for top, middle and bottom. The \emph{triple-crossing diagram} is a triple-crossing projection such that each of its triple point is a triple-crossing, such that $\pi^{-1}$ of the strand labeled $T$ (in the neighborhood of that triple point) is on the top of the strand corresponding to the strand labeled $M$, and the latter strand is on the top of the strand corresponding to the strand labeled $B$. (See Figure\;\ref{r01}.)

\begin{figure}[h!t]
    \centering
    \def\svgwidth{0.5\columnwidth}
\begingroup%
  \makeatletter%
  \providecommand\color[2][]{%
    \errmessage{(Inkscape) Color is used for the text in Inkscape, but the package 'color.sty' is not loaded}%
    \renewcommand\color[2][]{}%
  }%
  \providecommand\transparent[1]{%
    \errmessage{(Inkscape) Transparency is used (non-zero) for the text in Inkscape, but the package 'transparent.sty' is not loaded}%
    \renewcommand\transparent[1]{}%
  }%
  \providecommand\rotatebox[2]{#2}%
  \newcommand*\fsize{\dimexpr\f@size pt\relax}%
  \newcommand*\lineheight[1]{\fontsize{\fsize}{#1\fsize}\selectfont}%
  \ifx\svgwidth\undefined%
    \setlength{\unitlength}{323.24305733bp}%
    \ifx\svgscale\undefined%
      \relax%
    \else%
      \setlength{\unitlength}{\unitlength * \real{\svgscale}}%
    \fi%
  \else%
    \setlength{\unitlength}{\svgwidth}%
  \fi%
  \global\let\svgwidth\undefined%
  \global\let\svgscale\undefined%
  \makeatother%
  \begin{picture}(1,0.30507789)%
    \lineheight{1}%
    \setlength\tabcolsep{0pt}%
    \put(0,0){\includegraphics[width=\unitlength,page=1]{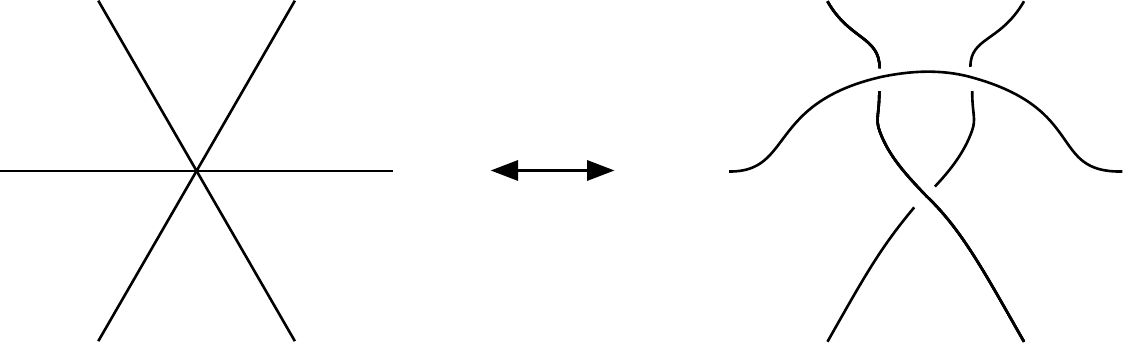}}%
    \put(0.24441612,0.10703206){\color[rgb]{0,0,0}\makebox(0,0)[lt]{\lineheight{1.25}\smash{\begin{tabular}[t]{l}$\scriptstyle T$\end{tabular}}}}%
    \put(0.17527425,0.02968691){\color[rgb]{0,0,0}\makebox(0,0)[lt]{\lineheight{1.25}\smash{\begin{tabular}[t]{l}$\scriptstyle M$\end{tabular}}}}%
    \put(0.08152262,0.0648438){\color[rgb]{0,0,0}\makebox(0,0)[lt]{\lineheight{1.25}\smash{\begin{tabular}[t]{l}$\scriptstyle B$\end{tabular}}}}%
  \end{picture}%
\endgroup%

    \caption{A deconstruction/construction of a triple-crossing.}
    \label{r01}
\end{figure}
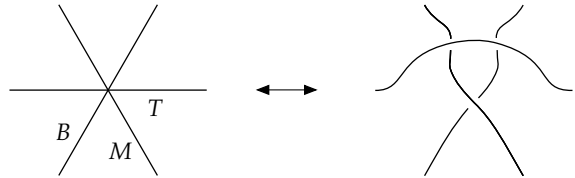

The \emph{triple-crossing number} of a knot or link $KL$, denoted $c_3(KL)$, is the least number of triple-crossings for any triple-crossing diagram of $KL$. The classical double-crossing number invariant we will denote by $c_2$. The \emph{minimal triple-crossing diagram} of a knot or link $KL$ is a triple-crossing diagrams of $KL$ that has exactly $c_3(KL)$ triple-crossings. 

\section{Moves on triple-crossing diagrams}\label{s6}

There are the following type of moves: ($T_1$) $1$-moves, ($T_2$) $2$-moves, ($T_3$) basepoint moves, ($T_4$) band moves, depicted in Figure\;\ref{r27}. The trivial pass move ($TP$) is presented in Figure\;\ref{r24} right. For the precise definitions of the moves we send the reader to \cite{AHP19}.

\begin{figure}[h!t]
    \centering
    \def\svgwidth{0.9\columnwidth}
    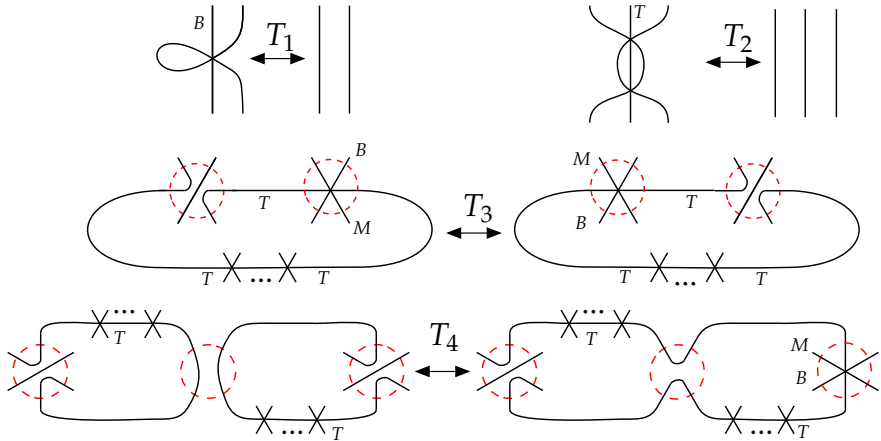
    \caption{Moves on triple-crossing diagrams}
    \label{r27}
\end{figure}

Denote a $J$-type move by one of the moves $J_R$ or its mirror $J'_R$ presented in Figure\;\ref{r24}, where $D_1$ and $D_2$ are arbitrary disjoint sub-diagrams of a diagram $D=D_1\cup D_2$. We can reduce the number of the four mentioned types of moves as follows.

\begin{figure}[h!t]
    \centering
    \def\svgwidth{0.9\columnwidth}
    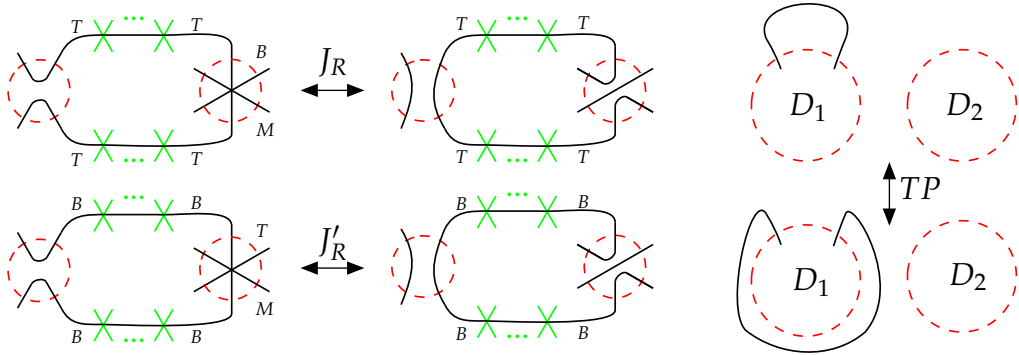
    \caption{$J$-type moves and a $TP$ move on triple-crossing diagrams.}
    \label{r24}
\end{figure}

\begin{theorem}\label{tw1}
Any sequence of moves $T_1,T_2,T_3,T_4$ can be realized by the $J$-type moves and a planar isotopy.
\end{theorem}

\begin{proof}
First, we define the $J_L$ move as shown in Figure\;\ref{r11} and denote $J_L'$ its mirror move. 
To prove the theorem, we will show that all (top, bottom, middle) interpretations of the type of moves $T_1$, $T_2$, $T_3$, $T_4$ are generated by moves $J_R$, $J_R'$, $J_L$, $J_L'$ (called in this Proof as $J$-type moves), and later that moves $J_L$, $J_L'$ are generated by $J_R$, $J_R'$.

\begin{figure}[h!t]
    \centering
    \def\svgwidth{.85\columnwidth}
    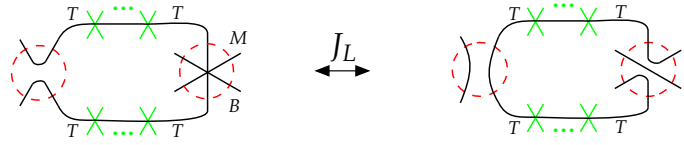
    \caption{Move $J_L$.}
    \label{r11}
\end{figure}

The $T_1$-type moves are realized by $J$-type moves by closing two strands to form a loop, and a planar isotopy as in Figure\;\ref{r14}. 

\begin{figure}[h!t]
    \centering
    \def\svgwidth{0.9\columnwidth}
    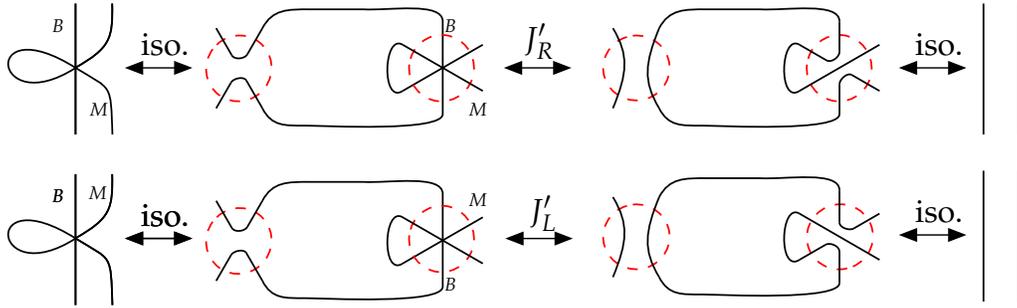
    \caption{Realizing a $T_1$-type move.}
    \label{r14}
\end{figure}

The $T_2$-type move is realized by $J$-type moves, and a planar isotopy as in Figure\;\ref{r18}. The $T_3$-type moves are realized by $J$-type moves and the $T_1$-type moves, and a planar isotopy as in Figure\;\ref{r19}, where only one case is considered. The other interpretations of the $T_3$-type moves are derived by analogy, also by $J$-type moves.
\par
It is straightforward that $T_4$-type moves can be realized by $J$-type moves, by performing it on just a part of the pictures (where the move change the image) as in Figure\;\ref{r12}. Finally, moves $J_L$, $J_L'$ are generated by $J_R$, $J_R'$ as shown in Figure\;\ref{r25} (in the $J_L$ case, the other move are derived analogously by the composition of mirror moves).

\begin{figure}[h!t]
    \centering
    \def\svgwidth{1\columnwidth}
    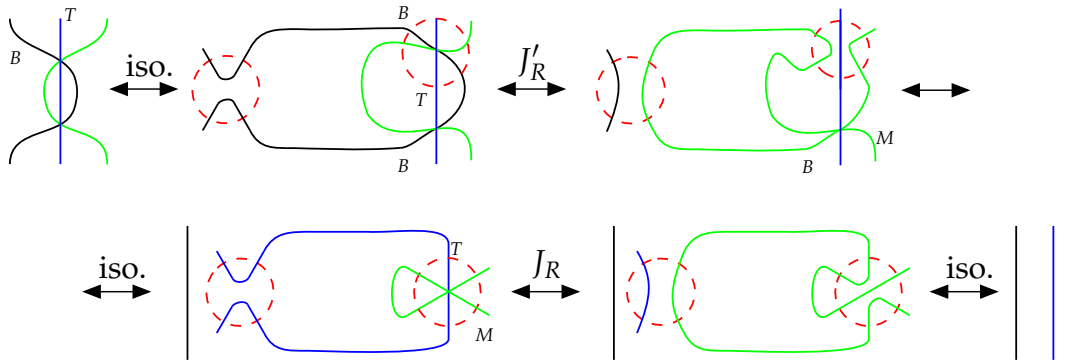
    \caption{Realizing a $T_2$-type move.}
    \label{r18}
\end{figure}

\begin{figure}[h!t]
    \centering
    \def\svgwidth{1.15\columnwidth}
    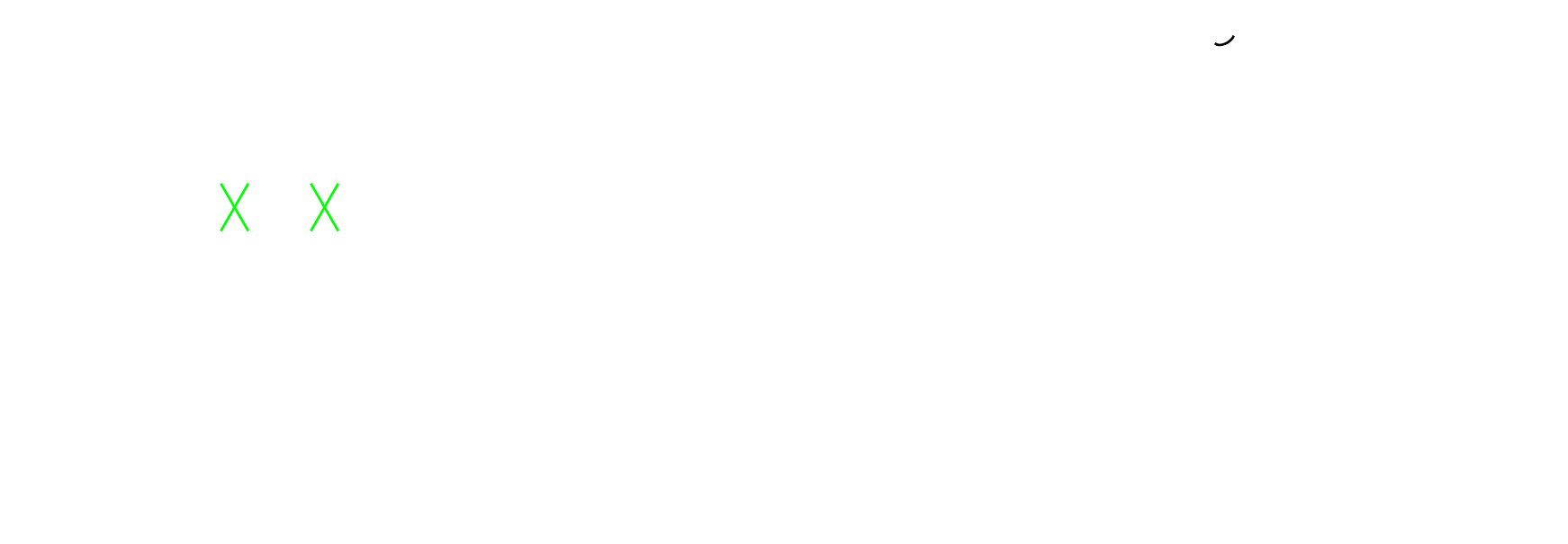
    \caption{Realizing a $T_3$-type move.}
    \label{r19}
\end{figure}

\begin{figure}[h!t]
    \centering
    \def\svgwidth{1\columnwidth}
    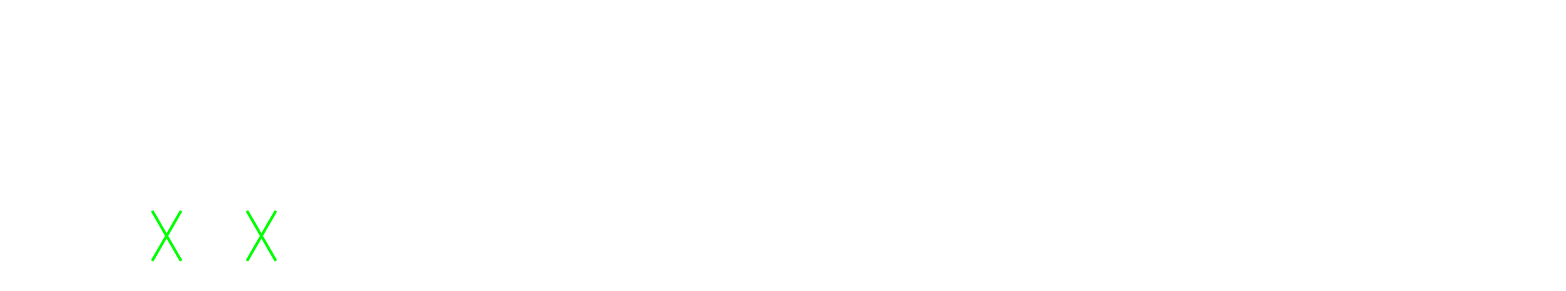
    \caption{Realizing a $T_4$-type move.}
    \label{r12}
\end{figure}

\begin{figure}[h!t]
    \centering
    \def\svgwidth{1.1\columnwidth}
    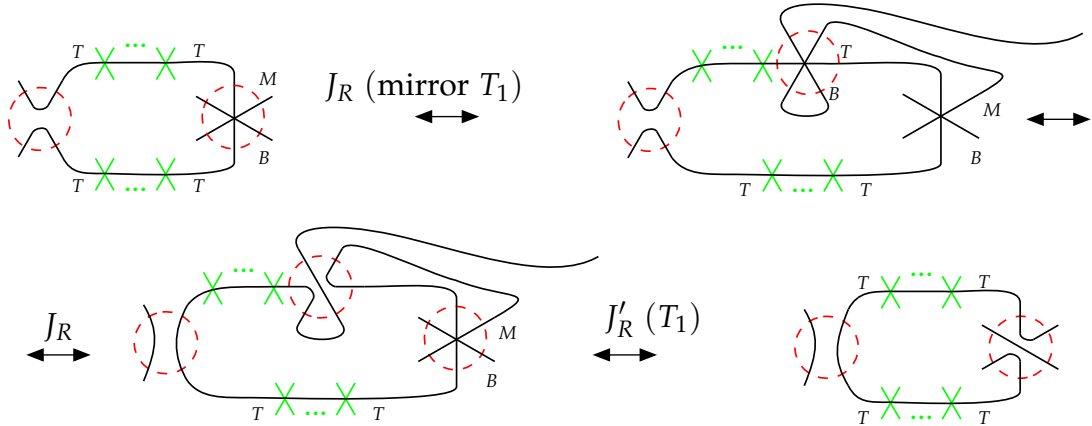
    \caption{Realizing a $J_L$-type move.}
    \label{r25}
\end{figure}

\end{proof}

\begin{remark}
In fact the $J$-type moves can be realized by $T_4$-type moves (and its mirror moves) and a planar isotopy, as shown in Figure\;\ref{r23}.
\end{remark}

\begin{figure}[h!t]
    \centering
    \def\svgwidth{1\columnwidth}
    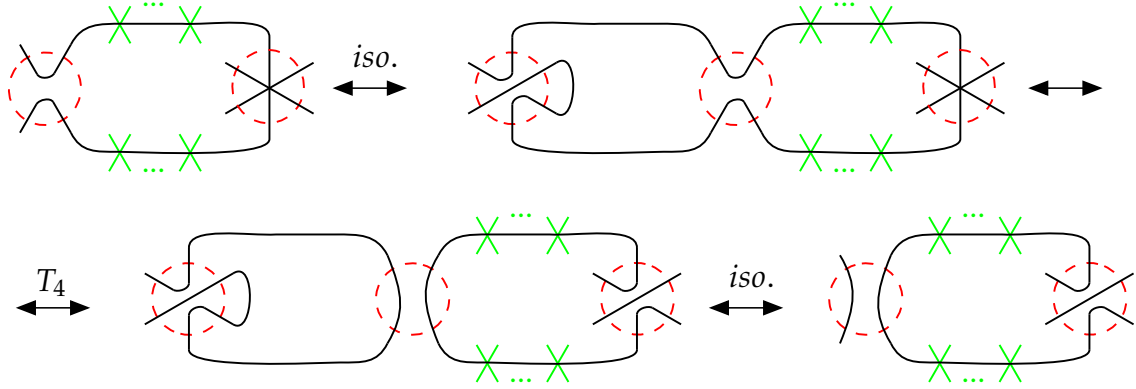
    \caption{Realizing a $J$-type move.}
    \label{r23}
\end{figure}

A \emph{natural orientation} (see an equivalent definition in \cite{AHP19}) on a triple-crossing diagram is an orientation of each component of that link, such that in each crossing the strands are oriented in-out-in-out-in-out, as we encircle the crossing. Theorem\;\ref{tw1} together with main theorems from \cite{AHP19} give us the following corollaries. 

\begin{corollary}
Two triple-crossing diagrams of knots are related by a sequence of $J$-type moves, if and only if they define the same knot type.
\end{corollary}

\begin{corollary}
Two unoriented triple-crossing diagrams $D_1$ and $D_2$ of nonsplit links are related by a sequence of $J$-type moves if and only if natural orientations on $D_1$ and $D_2$ define the same oriented link, up to complete reversal.
\end{corollary}

\begin{corollary}
Two unoriented triple-crossing diagrams $D_1$ and $D_2$ are related by a sequence of $\{J,TP\}$-type moves if and only if natural orientations on $D_1$ and $D_2$ define the same oriented link, up to reversal of maximal nonsplit sublinks.
\end{corollary}

The moves $J_R$ and $J_R'$ are independent of each another, because $J_R$ (as oppose to the $J_R'$) move changes the number of triple-crossings of name $eY$ in an $sPD$ code of the diagram (see Section\;\ref{s3}) shown in Figure\;\ref{r20}, therefore we have the following.

\begin{corollary}
The set $\{J_R, J_R'\}$ is a minimal generating set of moves connecting triple-crossing diagrams of the same knot.
\end{corollary}

\section{Generating triple-point projections}\label{s3}

\subsection{From classical projections to triple-point projections}

From the set of all pairwise non-isotopic (on a sphere) classical-crossing projections having $n$ crossings, we choose a set of projections (shadows) of connected, prime (i.e. it is not a connected sum of two diagrams) diagrams of classical links, without mirror images, without loops.
\par
We denote this set as $Sh_{n}$, and generate every element of it just as in one of the authors earlier papers (in \cite[Section\; 3.]{Jab19}). The enumeration of elements of this set (for an even $n$) is presented in Table\;\ref{table1} (second column), we associate each such projection with its $PD$ code (described in that paper and in \cite{katlas}).

\begin{figure}[h!t]
    \centering
    \def\svgwidth{0.79\columnwidth}
    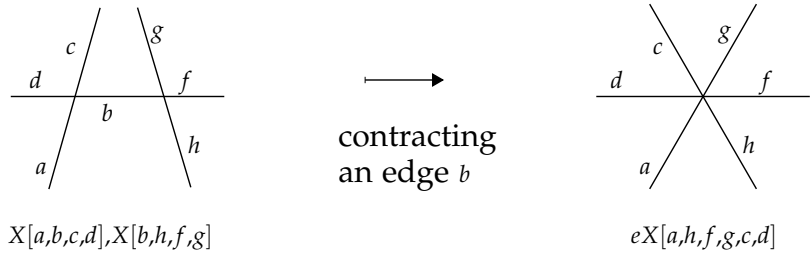
    \caption{The contraction operation on an edge}
    \label{r04}
\end{figure}

In the following theorem we are using the word "prime" in two different ways.

\begin{theorem}
The set $Ta_{n}$ of all prime, connected, triple-crossing projections of all minimal triple-crossing diagrams (up to mirror image) of prime knots and prime nonsplit links (with $n$ triple-crossing, $n>1$) can be generated from $Sh_{2n}$ as follows. Take a shadow from $Sh_{2n}$ and choose a set of its $n$ edges, such that each pair of the edges does not meet at a vertex. Perform on every selected edge the contraction operation (as shown in Figure\;\ref{r04}).
\end{theorem}

\begin{proof}
First, notice that if a triple-crossing projection (with more than one triple-crossing) is of a minimal triple-crossing diagram of some prime knot or prime nonsplit link, then it cannot have more than one loop in any triple-crossing. Otherwise, it will be either composite with Hopf link summand (when a two loops in the crossing are non-adjacent) or the triple-crossing can be removed (when a two loops in the crossing are adjacent, see Proof of Proposition\;\ref{prop3}) contradicting the minimality of the diagram.

\begin{figure}[h!t]
    \centering
    \def\svgwidth{0.75\columnwidth}
\begingroup%
  \makeatletter%
  \providecommand\color[2][]{%
    \errmessage{(Inkscape) Color is used for the text in Inkscape, but the package 'color.sty' is not loaded}%
    \renewcommand\color[2][]{}%
  }%
  \providecommand\transparent[1]{%
    \errmessage{(Inkscape) Transparency is used (non-zero) for the text in Inkscape, but the package 'transparent.sty' is not loaded}%
    \renewcommand\transparent[1]{}%
  }%
  \providecommand\rotatebox[2]{#2}%
  \newcommand*\fsize{\dimexpr\f@size pt\relax}%
  \newcommand*\lineheight[1]{\fontsize{\fsize}{#1\fsize}\selectfont}%
  \ifx\svgwidth\undefined%
    \setlength{\unitlength}{551.72457078bp}%
    \ifx\svgscale\undefined%
      \relax%
    \else%
      \setlength{\unitlength}{\unitlength * \real{\svgscale}}%
    \fi%
  \else%
    \setlength{\unitlength}{\svgwidth}%
  \fi%
  \global\let\svgwidth\undefined%
  \global\let\svgscale\undefined%
  \makeatother%
  \begin{picture}(1,0.22615973)%
    \lineheight{1}%
    \setlength\tabcolsep{0pt}%
    \put(0,0){\includegraphics[width=\unitlength,page=1]{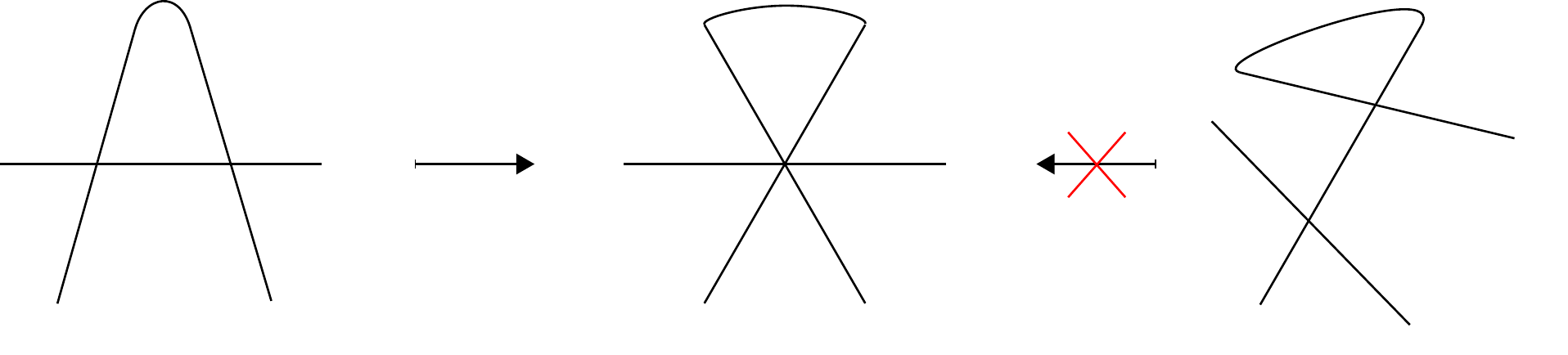}}%
    \put(0.08637692,0.00401086){\color[rgb]{0,0,0}\makebox(0,0)[lt]{\lineheight{1.25}\smash{\begin{tabular}[t]{l}$\scriptstyle p_d$\end{tabular}}}}%
    \put(0.4693255,0.00401086){\color[rgb]{0,0,0}\makebox(0,0)[lt]{\lineheight{1.25}\smash{\begin{tabular}[t]{l}$\scriptstyle p_t$\end{tabular}}}}%
    \put(0.81456581,0.00401086){\color[rgb]{0,0,0}\makebox(0,0)[lt]{\lineheight{1.25}\smash{\begin{tabular}[t]{l}$\scriptstyle p'_d$\end{tabular}}}}%
  \end{picture}%
\endgroup%

    \caption{A contraction to a loop.}
    \label{r21}
\end{figure}
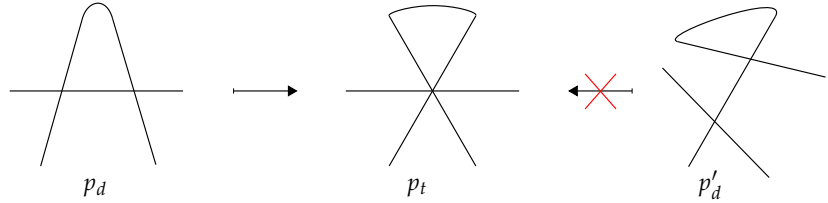

Assume now that a prime triple-crossing projection $p_t$ has some triple crossings each with only one loop. Then there is a prime double-crossing projection $p_d$ without loops that contracts to $p_t$. Each such triple-crossing loop can be obtained as in Figure\;\ref{r21}.
\par
It is straightforward that a contraction operation on a prime, connected graph results in a prime, connected graph. Assume now we have a triple-crossing projection $p_t$ that is prime and connected. It clearly cannot be generated from a disconnected projection by a contraction operation, but can be the result of contracting some edges from a composite double-crossing projection $p'_d$. Instead, there is a prime double-crossing projection $p_d$ that contracts to $p_t$, each such triple-crossing can be obtained as in Figure\;\ref{r22}.

\begin{figure}[h!t]
    \centering
    \def\svgwidth{0.75\columnwidth}
\begingroup%
  \makeatletter%
  \providecommand\color[2][]{%
    \errmessage{(Inkscape) Color is used for the text in Inkscape, but the package 'color.sty' is not loaded}%
    \renewcommand\color[2][]{}%
  }%
  \providecommand\transparent[1]{%
    \errmessage{(Inkscape) Transparency is used (non-zero) for the text in Inkscape, but the package 'transparent.sty' is not loaded}%
    \renewcommand\transparent[1]{}%
  }%
  \providecommand\rotatebox[2]{#2}%
  \newcommand*\fsize{\dimexpr\f@size pt\relax}%
  \newcommand*\lineheight[1]{\fontsize{\fsize}{#1\fsize}\selectfont}%
  \ifx\svgwidth\undefined%
    \setlength{\unitlength}{570.06867147bp}%
    \ifx\svgscale\undefined%
      \relax%
    \else%
      \setlength{\unitlength}{\unitlength * \real{\svgscale}}%
    \fi%
  \else%
    \setlength{\unitlength}{\svgwidth}%
  \fi%
  \global\let\svgwidth\undefined%
  \global\let\svgscale\undefined%
  \makeatother%
  \begin{picture}(1,0.45231957)%
    \lineheight{1}%
    \setlength\tabcolsep{0pt}%
    \put(0,0){\includegraphics[width=\unitlength,page=1]{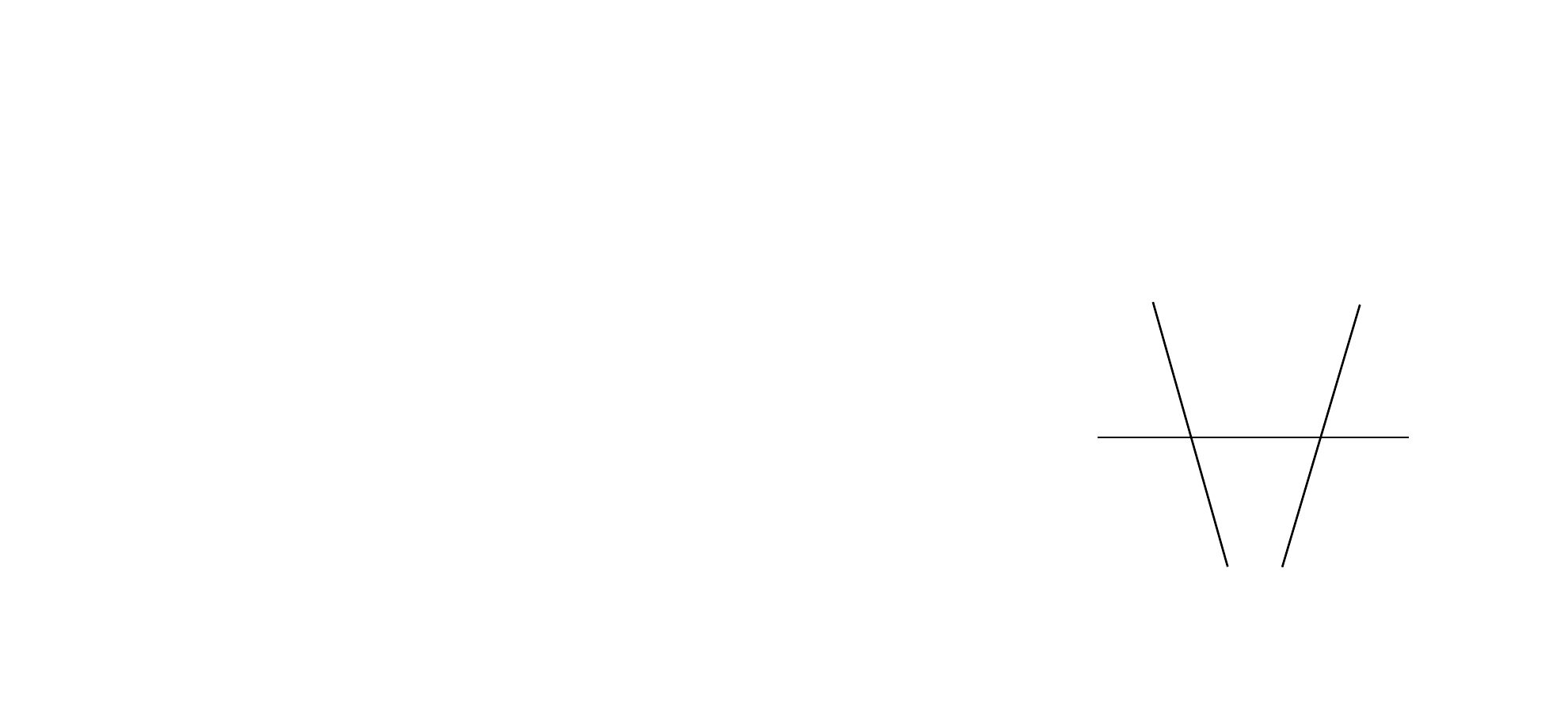}}%
    \put(0.08543062,0.0038818){\color[rgb]{0,0,0}\makebox(0,0)[lt]{\lineheight{1.25}\smash{\begin{tabular}[t]{l}$\scriptstyle p_d$\end{tabular}}}}%
    \put(0.49172753,0.0038818){\color[rgb]{0,0,0}\makebox(0,0)[lt]{\lineheight{1.25}\smash{\begin{tabular}[t]{l}$\scriptstyle p_t$\end{tabular}}}}%
    \put(0.80653654,0.0038818){\color[rgb]{0,0,0}\makebox(0,0)[lt]{\lineheight{1.25}\smash{\begin{tabular}[t]{l}$\scriptstyle p'_d$\end{tabular}}}}%
    \put(0,0){\includegraphics[width=\unitlength,page=2]{LM22.pdf}}%
  \end{picture}%
\endgroup%

    \caption{A contraction to a prime projection.}
    \label{r22}
\end{figure}

\end{proof}

After this contraction operations, we associate to each triple-point projection from $Ta_{n}$ its \emph{sPD code} (a six-valent Planar Diagram code), consisting of $n$ elements from the underlying classical $PD$ code consisting of $2n$ elements, by first removing the labels of every contracting edge and then merging two codes for the endpoints of this edge (with the new name $eX$ of the vertex), as shown in Figure\;\ref{r04}.

\begin{figure}[h!t]
    \centering
    \def\svgwidth{1.2\columnwidth}
    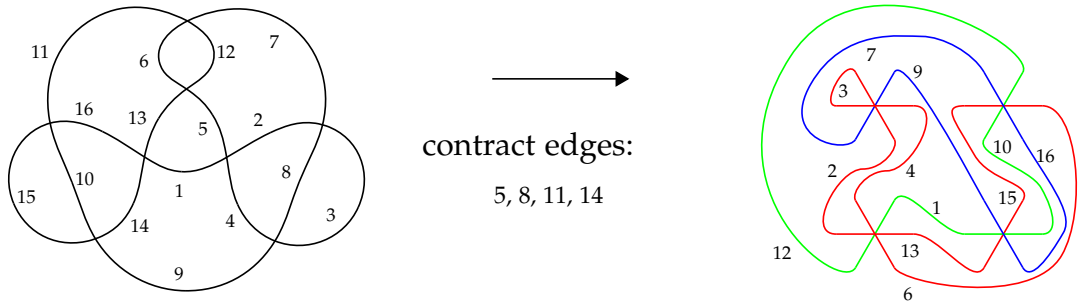
    \caption{An example of transformation from a classical projection to a triple-point projection}
    \label{r17}
\end{figure}

Then, we re-numerate each $sPD$ code with consecutive positive integers from one up to three times the number of triple-crossings, with increasing labels as we go around each transverse circular component in the triple-projection. An example is presented in Figure\;\ref{r17}, where a knot diagram with code $\scriptstyle PD[X[1,4,2,5],X[3,8,4,9],X[12,6,13,5],X[13,16,14,1],X[9,14,10,15],X[15,10,16,11],X[6,12,7,11],X[7,2,8,3]]$ is transformed by contracting four edges labeled: $5,8,11,14$, to a three-component link diagram with code $$sPD[eX[1,4,2,12,6,13],eX[4,9,3,3,7,2],eX[15,10,16,6,12,7],eX[1,13,16,10,15,9]].$$ The enumeration of elements of the set $Ta_{n}$ is presented in Table\;\ref{table1} (third column).

\pagebreak

\begin{center}
\begin{longtable}{@{}r||r|r|r|r|r|r|r|r|r@{}}
\caption{Enumeration of shadows and diagrams.\label{table1}}\\
	$n$&$Sh_{2n}$&$Ta_{n}$&$Tb_{n}$&$Gr_{n}$&$Th^0_{n}$&$Th_{n}$&$TD_{n}$&$K_{n}$&$L_{n}$\\
	\hline
2&2&14&4&2&1&3&108&2&2\\
3&9&108&18&4&1&9&1.944&2&11\\
4&62&1.312&222&20&5&57&73.872&24&70\\
5&803&29.198&&&12&&&$\geq$ 94&$\geq$ 87\\
6&15.882&1.007.494&&&&&&&\\
	\hline
\end{longtable}
\end{center}

\subsection{Checking an isotopy type and loop moves}

Now, we create a set $Tb_{n}$ of those projections from $Ta_{n}$, that each pair of shadows from $Tb_{n}$ are neither isotopic nor is one of them isotopic to the mirror image of the other. For our algorithm, we use the following (well-known) general method for testing spherical graphs isotopy (see \cite{WalLeh72}).

\begin{figure}[h!t]
    \centering
    \def\svgwidth{1.15\columnwidth}
    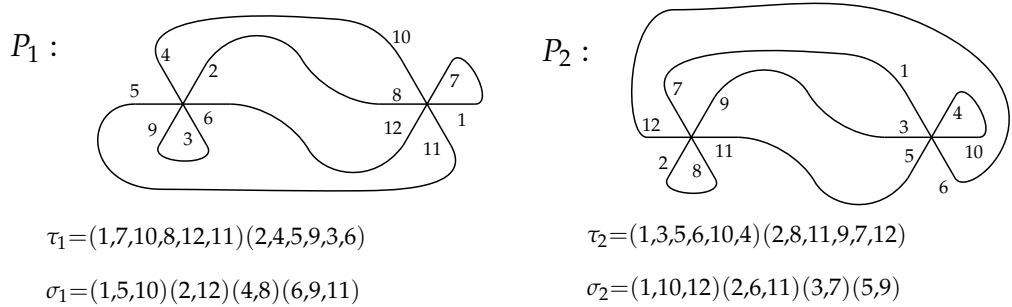
    \caption{An example of isotopic projections on the sphere.}
    \label{r10}
\end{figure}

Any isotopy class of connected graphs with $e$ number of edges embedded in the sphere corresponds to a pair of permutations $\sigma,\tau$
on $2e$ elements, such that $\sigma\cdot\tau$ has order two, and $\sigma,\tau$ are determined up to simultaneous conjugacy. The correspondence is as follows. First, we label each side of each edge of a graph with a distinct number from $1$ to $2e$ arbitrarily.
\par
To define $\tau$, we make a cycle for each vertex $v$, which involves only the clockwise sides of each edge adjacent to $v$, and rotates each one step counterclockwise. To define $\sigma$, we make a cycle for each face $f$, which rotates all the edge labels that $f$ encounters each step counterclockwise (from $f$'s point of view). From the permutations $\sigma,\tau$, we can reconstruct the isotopy class by gluing together the faces, edges, and vertices according to the permutations. In the examples in Figure\;\ref{r10} the projections $P_1$ and $P_2$ are isotopic by a conjugacy permutation $s=(1, 6, 5, 4, 3, 2) (7, 12, 11, 10, 9, 8)$. The enumeration of elements of the set $Tb_{n}$ is presented in Table\;\ref{table1} (fourth column).

\begin{lemma}[\cite{Ada13}]\label{lem01}
Any clasp (with two classical crossings) can be locally transformed (by using an ambient isotopy in $\mathbb{R}^3$) to a one triple-crossing, as in Figure\;\ref{r03}.
\end{lemma}

\begin{figure}[h!t]
    \centering
    \def\svgwidth{0.5\columnwidth}
    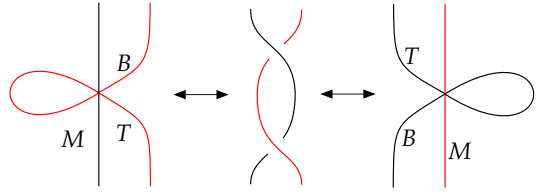
    \caption{Exchanging two classical crossing with one triple-crossing.}
    \label{r03}
\end{figure}

\begin{proposition}\label{prop3}
Local moves $M1$, $M2$ (see Figure\;\ref{r02}.) on a triple-crossing projection, do not change a knot or a link type of any triple-crossing diagram having that projection (with a proper $T,M,B$ labeling after each move).
\end{proposition}

\begin{figure}[h!t]
    \centering
    \def\svgwidth{0.5\columnwidth}
\begingroup%
  \makeatletter%
  \providecommand\color[2][]{%
    \errmessage{(Inkscape) Color is used for the text in Inkscape, but the package 'color.sty' is not loaded}%
    \renewcommand\color[2][]{}%
  }%
  \providecommand\transparent[1]{%
    \errmessage{(Inkscape) Transparency is used (non-zero) for the text in Inkscape, but the package 'transparent.sty' is not loaded}%
    \renewcommand\transparent[1]{}%
  }%
  \providecommand\rotatebox[2]{#2}%
  \newcommand*\fsize{\dimexpr\f@size pt\relax}%
  \newcommand*\lineheight[1]{\fontsize{\fsize}{#1\fsize}\selectfont}%
  \ifx\svgwidth\undefined%
    \setlength{\unitlength}{334.68651496bp}%
    \ifx\svgscale\undefined%
      \relax%
    \else%
      \setlength{\unitlength}{\unitlength * \real{\svgscale}}%
    \fi%
  \else%
    \setlength{\unitlength}{\svgwidth}%
  \fi%
  \global\let\svgwidth\undefined%
  \global\let\svgscale\undefined%
  \makeatother%
  \begin{picture}(1,0.61519232)%
    \lineheight{1}%
    \setlength\tabcolsep{0pt}%
    \put(0,0){\includegraphics[width=\unitlength,page=1]{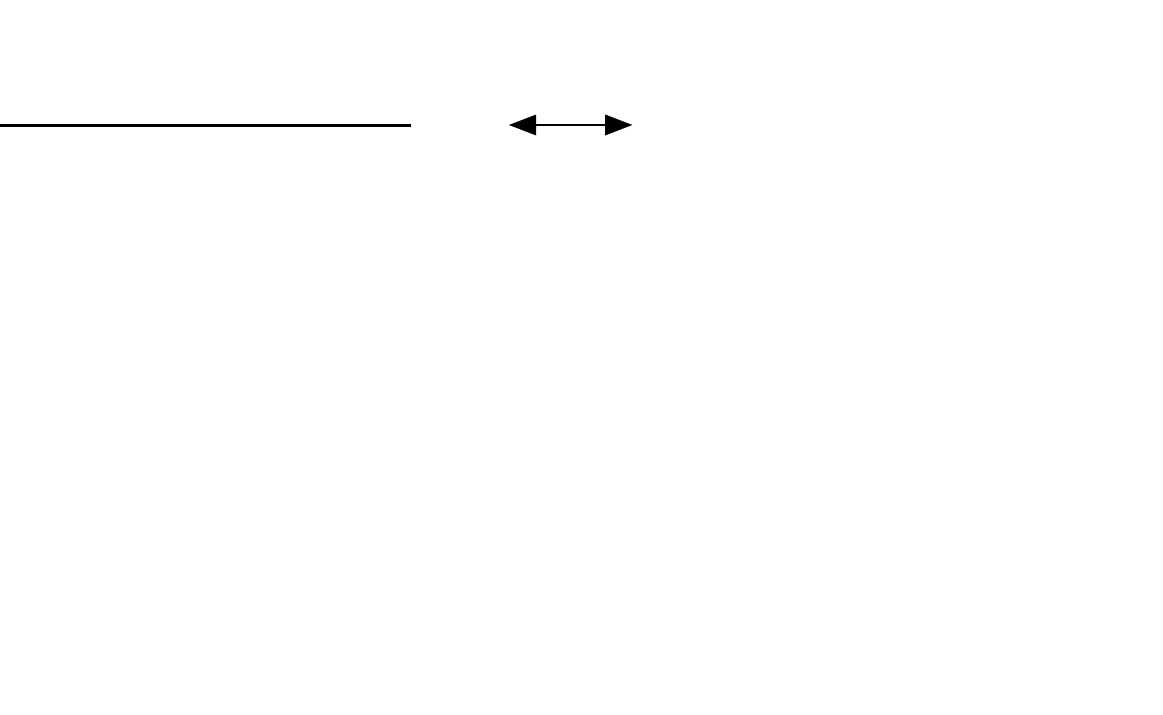}}%
    \put(0.47213007,0.53735322){\color[rgb]{0,0,0}\makebox(0,0)[lt]{\lineheight{1.25}\smash{\begin{tabular}[t]{l}$M1$\end{tabular}}}}%
    \put(0,0){\includegraphics[width=\unitlength,page=2]{LM02.pdf}}%
    \put(0.47461184,0.25952841){\color[rgb]{0,0,0}\makebox(0,0)[lt]{\lineheight{1.25}\smash{\begin{tabular}[t]{l}$M2$\end{tabular}}}}%
    \put(0,0){\includegraphics[width=\unitlength,page=3]{LM02.pdf}}%
  \end{picture}%
\endgroup%

    \caption{Moves $M1$ and $M2$.}
    \label{r02}
\end{figure}

\begin{proof}
In the $M_1$ case of having two adjacent loops in a triple-crossing diagram of a knot or link attached to one vertex, the four endpoints of the loops are adjacent, therefore one of them is labeled $M$, which allows us to make a type $T_1$ move (or its mirror move) to reduce that triple-crossing. In the $M_2$ case, if one of the strands forming a loop has label $M$ then we make first a type $T_1$ move (or its mirror move) to reduce the triple-crossing and then we make another move creating one triple-crossing with a loop in the opposite side. If the strands forming a loop are labeled $B, T$ then we make a move as in Lemma\;\ref{lem01} and push the loop to the other side of the triple-crossing with proper $T,M,B$ labeling after the move.
\end{proof}

Computational results led us to the following tabulation of knots and links triple-projections.

\begin{proposition}\label{projections}
All prime, connected, triple-crossing projections, up to spherical isotopy, up to mirror image and up to moves $M1$, $M2$, with one, two and three triple-points are presented in Figure\;\ref{r06}. The case for four triple-points is presented in Figures\;\ref{r15}--\ref{r13}. The enumeration of elements of this set, denoted $Th_{n}$, is presented in Table\;\ref{table1} (seventh column).
\end{proposition}

\begin{figure}[h!t]
    \centering
		\caption{Table of projections of $1$,$2$ or $3$ triple points.}
    \def\svgwidth{0.85\columnwidth}
    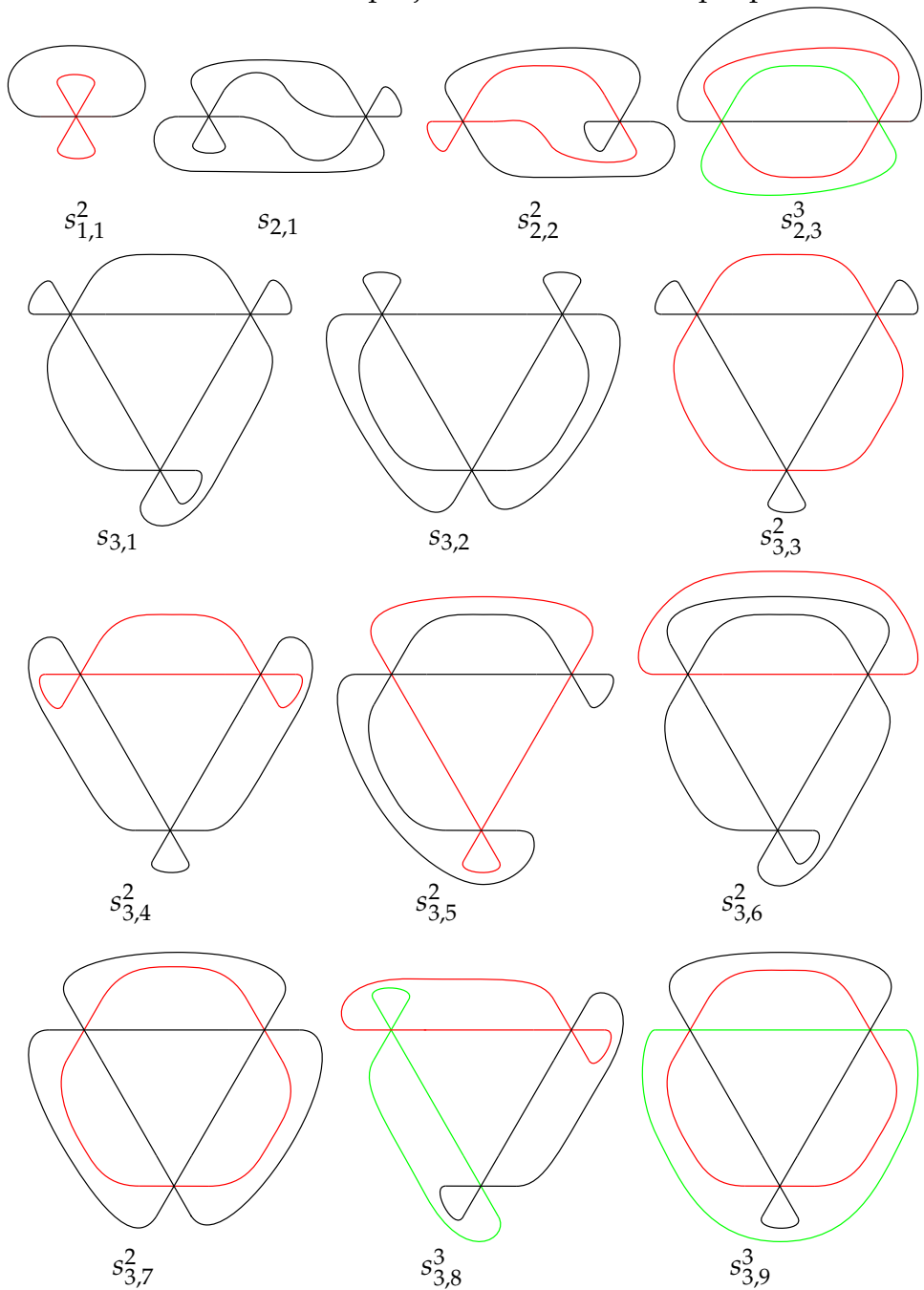
    \label{r06}
\end{figure}

The enumeration of elements of the subset of the set $Th_{n}$ consisting only of projections without loops, denoted $Th^0_{n}$, is presented in Table\;\ref{table1} (sixth column).

\subsection{Dual graphs and projections}

Define $\mathcal{G}$ as a set of all planar, simple, connected graph embeddings, which vertices are of degree at least two, such that each face is a quadrangle or a hexagon (including the external face). If such a graph has $q$ quadrangles and $h$ hexagons we label it $g^{q,h}_n$, where $n$ is the index of the graph generated within the family of graphs with the same pair $q,h$.
\par
We can generate all prime, connected, triple-crossing projections (up to mirror image) of all minimal triple-crossing diagrams of knots and links (with more than one triple-crossing) as follows. First, we take the dual graph to each graph from the set $\mathcal{G}$ and then we put one loop on each vertex of the valency four. 
\par
This method is useful in a tabulation of projections as in Theorem\;\ref{projections}, but we cannot find known tabulations (or a generating function for an enumeration) of these spherical partitions. Instead, we use our earlier method of generating triple-crossing projections to produce such a table for the following cases.

\begin{proposition}
All elements of $\mathcal{G}$ with two or three faces are presented in Figure\;\ref{r07}. The case of four-face graphs is presented in Figure\;\ref{r08}. The enumeration of elements of the set $G$ with $n$ faces, denoted $Gr_n$ is presented in Table\;\ref{table1} (fifth column).
\end{proposition}

Elements of the subset of the set $Gr_{5}$, consisting only of projections with hexagons, are presented in Figure\;\ref{r09}. The dual graphs to this graphs are projections forming the set $Th^0_{5}$.

\begin{figure}[h!t]
    \centering
    \def\svgwidth{\columnwidth}
\begingroup%
  \makeatletter%
  \providecommand\color[2][]{%
    \errmessage{(Inkscape) Color is used for the text in Inkscape, but the package 'color.sty' is not loaded}%
    \renewcommand\color[2][]{}%
  }%
  \providecommand\transparent[1]{%
    \errmessage{(Inkscape) Transparency is used (non-zero) for the text in Inkscape, but the package 'transparent.sty' is not loaded}%
    \renewcommand\transparent[1]{}%
  }%
  \providecommand\rotatebox[2]{#2}%
  \newcommand*\fsize{\dimexpr\f@size pt\relax}%
  \newcommand*\lineheight[1]{\fontsize{\fsize}{#1\fsize}\selectfont}%
  \ifx\svgwidth\undefined%
    \setlength{\unitlength}{615.30530955bp}%
    \ifx\svgscale\undefined%
      \relax%
    \else%
      \setlength{\unitlength}{\unitlength * \real{\svgscale}}%
    \fi%
  \else%
    \setlength{\unitlength}{\svgwidth}%
  \fi%
  \global\let\svgwidth\undefined%
  \global\let\svgscale\undefined%
  \makeatother%
  \begin{picture}(1,0.19686296)%
    \lineheight{1}%
    \setlength\tabcolsep{0pt}%
    \put(0.87800137,0.00369243){\color[rgb]{0,0,0}\makebox(0,0)[lt]{\lineheight{1.25}\smash{\begin{tabular}[t]{l}$g^{1,2}_1$\end{tabular}}}}%
    \put(0,0){\includegraphics[width=\unitlength,page=1]{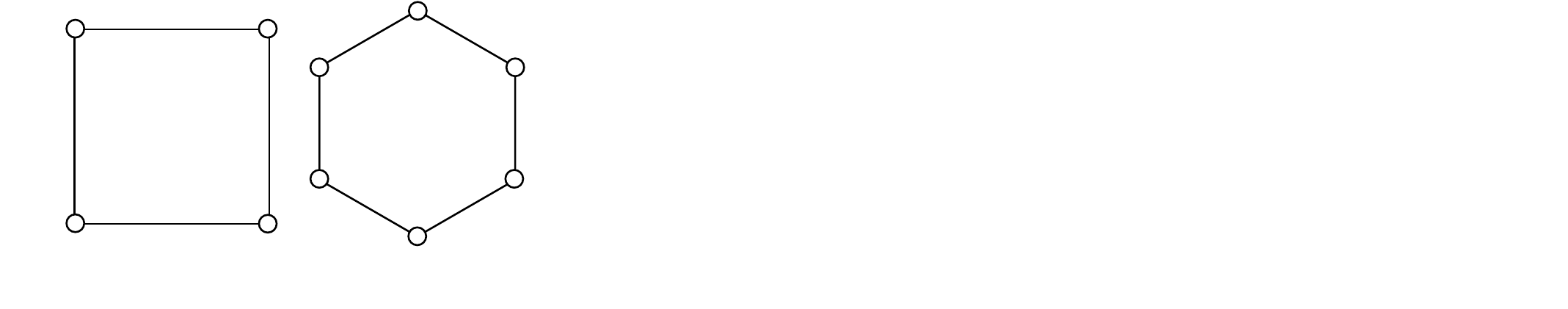}}%
    \put(0.09107582,0.00369243){\color[rgb]{0,0,0}\makebox(0,0)[lt]{\lineheight{1.25}\smash{\begin{tabular}[t]{l}$g^{2,0}_1$\end{tabular}}}}%
    \put(0.24846094,0.00369243){\color[rgb]{0,0,0}\makebox(0,0)[lt]{\lineheight{1.25}\smash{\begin{tabular}[t]{l}$g^{0,2}_1$\end{tabular}}}}%
    \put(0.40584606,0.00369243){\color[rgb]{0,0,0}\makebox(0,0)[lt]{\lineheight{1.25}\smash{\begin{tabular}[t]{l}$g^{3,0}_1$\end{tabular}}}}%
    \put(0.5632311,0.00369243){\color[rgb]{0,0,0}\makebox(0,0)[lt]{\lineheight{1.25}\smash{\begin{tabular}[t]{l}$g^{2,1}_1$\end{tabular}}}}%
    \put(0.72061625,0.00369243){\color[rgb]{0,0,0}\makebox(0,0)[lt]{\lineheight{1.25}\smash{\begin{tabular}[t]{l}$g^{0,3}_1$\end{tabular}}}}%
    \put(0,0){\includegraphics[width=\unitlength,page=2]{LM07.pdf}}%
    \put(-0.00135302,0.17608072){\color[rgb]{0,0,0}\makebox(0,0)[lt]{\lineheight{1.25}\smash{\begin{tabular}[t]{l}$\mbox{\ }$\end{tabular}}}}%
  \end{picture}%
\endgroup%

    \caption{Spherical graphs from $\mathcal{G}$ with two or three faces.}
    \label{r07}
\end{figure}

\section{Identification of knots and links}\label{s4}
\subsection{From triple-point projections to triple-point diagrams}

From each $sPD$ code form $Th_{n}$, we then generate $6^{n}$ possible $sPD$ codes by choosing the $\{T, M, B\}$ to be changed by its type among six possibilities of permuting three-letter set. The resulting set of all such diagrams we define as $TD_{n}$. 
\par
Then, we change in an $sPD$ code the crossing name, from $eX$ to $eY$ if we encounter the combination $B,M,T$ as we travel clockwise through the edges form the crossing. We rotate the position of numbers (edges) in each crossing, such that the first number corresponds to the incoming edge from the bottom (according to the orientation inherited from the monotonicity of edge labels), see Figure\;\ref{r20}.   

\begin{figure}[h!t]
    \centering
    \def\svgwidth{0.65\columnwidth}
    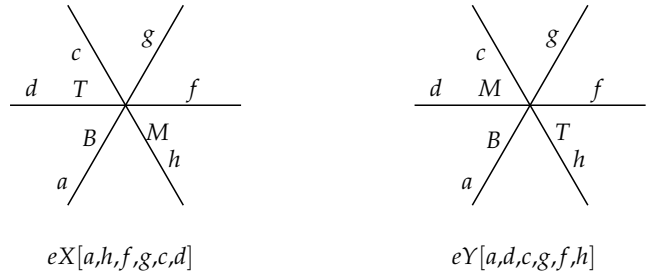
    \caption{Code crossing names.}
    \label{r20}
\end{figure}

\begin{figure}[h!t]
    \centering
    \def\svgwidth{1.05\columnwidth}
    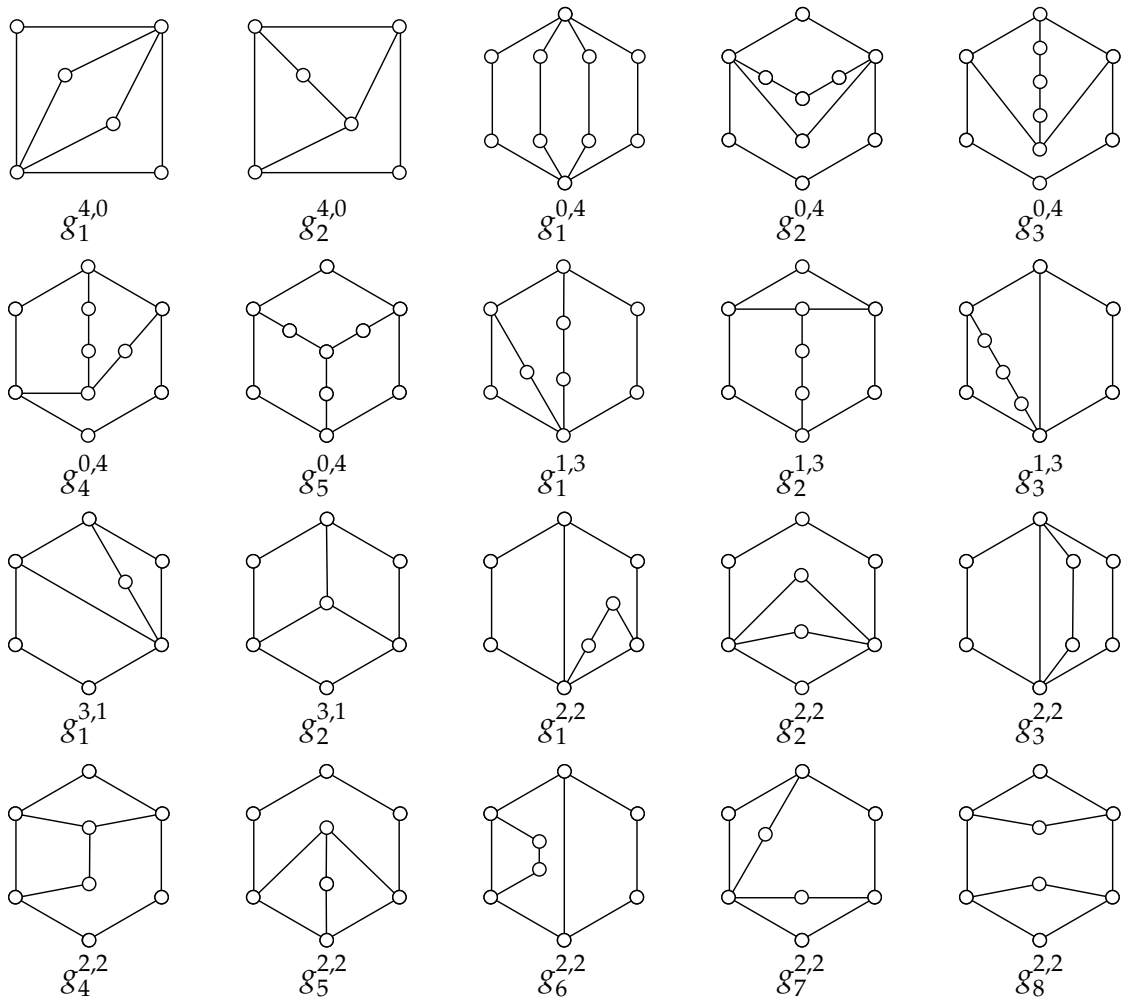
    \caption{Spherical graphs from $\mathcal{G}$ with four faces.}
    \label{r08}
\end{figure}

By the analogous argument about $PD$ codes as is in the paper \cite{Mat15}, we can derive the following. 

\begin{proposition}
Given an $sPD$ code we can produce a well-defined link diagram on a sphere.
\end{proposition}

The numbers in the eighth column of Table\;\ref{table1} counts the number of all created from the set $Th_{n}$ marked diagrams in amounts obtained by the equality $|TD_{n}|=|Th_{n}|\cdot 6^{n}$.

\subsection{Distinguishing knots and links}

To identify types of knots and links, we use the two-variable Kauffman polynomial, which turns out to be strong enough for our purpose. In some cases of diagrams, we also use a program to reduce the upper bound on the double-crossing number to eliminate repetition on the polynomial with higher double-crossing knots. 
\par
The enumeration of the set of all prime knots $K$ with $c_3(K)=n$, denoted $K_n$, is presented (up to mirror image) in Table\;\ref{table1} (ninth column). The case for prime links ($L_n$) is presented in the last column of the table. Names and triple-crossing numbers of prime knots is presented in Table\;\ref{table3}, the case for prime links is presented in Table\;\ref{table6}.

\begin{center}
\renewcommand{\arraystretch}{1.3}
\begin{longtable}{@{}r||l@{}}
\caption{Triple-crossing number of prime knots.\label{table3}}\\
	$c_3$&knots\\
	\hline
		$1$&-\\
		\hline
		$2$&$3_1, 4_1$.\\
		\hline
		$3$&$5_2, 6_1$.\\
		\hline
		$4$&$5_1, 6_2, 6_3, 7_2, 7_4, 7_6, 7_7, 8_1, 8_3, 8_{12}, 8_{20}, 8_{21}, 9_{42}, 9_{44}, 9_{45}, 9_{46}, 9_{48},10_{132}, 10_{136}, $\\
		&$10_{137}, 10_{140}, K11n_{38}, K11n_{139}, K12n_{462}$.\\
		\hline
		$5$&$7_3, 7_5, 8_4, 8_6, 8_8, 8_{11}, 8_{13}, 8_{14}, 8_{15}, 9_2, 9_5, 9_8, 9_{12}, 9_{14}, 9_{15}, 9_{19}, 9_{21}, 9_{25}, 9_{35}$,$9_{37}$,\\
		&$9_{39}, 9_{41}, 9_{49}, 10_{1}, 10_3, 10_{13}$, $10_{35}$, $10_{58}$, $10_{129}$, $10_{130}$, $10_{131}$, $10_{133}$, $10_{135}$, $10_{144}$,\\ 
		&$10_{145}$, $10_{146}$, $10_{147}$, $10_{162}$, $10_{164}$, $10_{165}$, $K11n_{1}$, $K11n_{3}$, $K11n_{12}$, $K11n_{17}$,\\
		& $K11n_{18}$, $K11n_{20}$, $K11n_{28}$, $K11n_{29}$, $K11n_{49}$, $K11n_{62}$, $K11n_{63}$, $K11n_{68}$, $K11n_{79}$, \\
		&$K11n_{83}$, $K11n_{91}$, $K11n_{100}$, $K11n_{101}$, $K11n_{102}$, $K11n_{113}$, $K11n_{114}$, $K11n_{116}$,\\
		&$K11n_{117}$, $K11n_{123}$, $K11n_{132}$, $K11n_{140}$, $K11n_{141}$, $K11n_{142}$, $K11n_{162}$, $K11n_{170}$,\\
		&$K12n_{11}$, $K12n_{25}$, $K12n_{46}$, $K12n_{48}$, $K12n_{49}$, $K12n_{65}$, $K12n_{121}$, $K12n_{145}$,\\
		&$K12n_{200}$, $K12n_{274}$, $K12n_{282}$, $K12n_{310}$, $K12n_{311}$, $K12n_{323}$, $K12n_{358}$, $K12n_{359}$,\\
		&$K12n_{414}$, $K12n_{442}$, $K12n_{443}$, $K12n_{452}$, $K12n_{478}$, $K12n_{523}$, $K12n_{582}$, $K12n_{608}$,\\
		&$K12n_{838}$, $(\ldots)$\\
				\hline	
$6$&$K11a_{13}$, $K11a_{59}$, $K11a_{61}$, $K11a_{65}$, $K11a_{98}$, $K11a_{103}$, $K11a_{145}$, $K11a_{166}$, \\ 
&$K11a_{195}$, $K11a_{201}$, $K11a_{209}$, $K11a_{210}$, $K11a_{211}$, $K11a_{214}$, $K11a_{218}$, $K11a_{219}$,\\ 
&$K11a_{226}$, $K11a_{228}$, $K11a_{229}$, $K11a_{230}$, $K11a_{247}$, $K11a_{343}$, $K12a_{197}$, $K12a_{482}$,\\ 
&$K12a_{690}$, $K12a_{691}$, $K12a_{803}$, $K12a_{1124}$, $K12a_{1166}$, $K12a_{1202}$, $K12a_{1287}$, $(\ldots)$\\
	\hline
$7$&$K14a_{12741}$, $K14a_{16442}$, $K14a_{17385}$, $K14a_{17730}$, $K14a_{18053}$, $K14a_{19429}$, $(\ldots)$\\
\end{longtable}
\end{center}

\begin{center}
\renewcommand{\arraystretch}{1.3}
\begin{longtable}{@{}r||l@{}}
\caption{Triple-crossing number of prime links.\label{table6}}\\
	$c_3$&links\\
\hline
\endfirsthead
\multicolumn{2}{c}
{\tablename\ \thetable\ -- \textit{Continued from previous page}} \\
	$c_3$&links\\
\hline
\endhead
\hline \multicolumn{2}{r}{\textit{Continued on next page}} \\
\endfoot
\hline
\endlastfoot
		$1$&$2^2_1.$\\
		\hline
		$2$&$4^2_1,6^3_3$.\\
		\hline
		$3$&$5^2_1,6^2_1,6^2_3,6^3_1, 7^2_7, 7^2_8$, $8^2_{15}, 8^2_{16},8^3_7,8^3_8,9^2_{49}$.\\
		\hline
		$4$&$6^2_2$, $6^3_2$, $7^2_1$, $7^2_2$, $7^2_3$, $7^2_5$, $7^3_1$, $8^2_1$, $8^2_3$, $8^2_6$, $8^2_9$, $8^3_1$, $8^3_3$, $8^3_4$, $8^3_9$, $8^3_{10}$, $8^4_1$, $8^4_2$, $8^4_3$, $9^2_{45}$, $9^2_{46}$,\\
		&$9^2_{47}$, $9^2_{48}$, $9^2_{50}$, $9^2_{52}$, $9^2_{53}$, $9^2_{54}$, $9^3_{13}$, $9^3_{14}$, $9^3_{18}$, $9^3_{19}$, $L10n_7$, $L10n_8$, $L10n_{9}$, $L10n_{19}$,\\
		&$L10n_{31}$, $L10n_{45}$, $L10n_{48}$, $L10n_{49}$, $L10n_{68}$, $L10n_{69}$, $L10n_{70}$, $L10n_{71}$, $L10n_{77}$,\\
		&$L10n_{78}$, $L10n_{81}$, $L10n_{92}$, $L10n_{93}$, $L10n_{94}$, $L10n_{104}$, $L10n_{105}$, $L10n_{108}$,\\
		&$L10n_{109}$, $L11n_{140}$, $L11n_{141}$, $L11n_{204}$, $L11n_{376}$, $L11n_{378}$, $L11n_{419}$, $L11n_{420}$,\\
		&$L12n_{1804}$, $L12n_{1806}$, $L12n_{1807}$, $L12n_{1997}$, $L12n_{1998}$, $L12n_{2150}$, $L12n_{2151}$,\\
		&$L12n_{2159}$, $L12n_{2206}$, $L12n_{2209}$.\\
		\hline
		$5$&$7^2_4$, $7^2_6$, $8^2_2$, $8^2_4$, $8^2_5$, $8^2_7$, $8^2_{8}$, $8^2_{10}$, $8^2_{11}$, $8^2_{12}$, $8^2_{13}$, $8^2_{14}$, $8^3_{2}$, $8^3_{5}$, $8^3_6$, $9^2_{1}$, $9^2_{2}$, $9^2_{5}$, $9^2_{8}$, $9^2_{10}$, \\
		&$9^2_{11}$, $9^2_{12}$, $9^2_{16}$, $9^2_{17}$, $9^2_{20}$, $9^2_{23}$, $9^2_{24}$, $9^2_{25}$, $9^2_{26}$, $9^2_{28}$, $9^2_{30}$, $9^2_{40}$, $9^2_{41}$, $9^2_{57}$, $9^2_{58}$, $9^2_{60}$, $9^2_{61}$, \\
		&$9^3_{1}$, $9^3_{2}$, $9^3_{4}$, $9^3_{5}$, $9^3_{6}$, $9^3_{7}$, $9^3_{8}$, $9^3_{15}$, $9^3_{16}$, $9^3_{17}$, $9^4_{1}$, $L10a_{12}$, $L10a_{16}$, $L10a_{37}$, $L10a_{40}$,\\
		&$L10a_{48}$, $L10a_{67}$, $L10a_{75}$, $L10a_{87}$, $L10a_{98}$, $L10a_{102}$, $L10a_{118}$, $L10a_{123}$, $L10a_{126}$,\\
		&$L10a_{131}$, $L10a_{133}$, $L10a_{139}$, $L10a_{145}$, $L10a_{153}$ $L10a_{154}$, $L10a_{161}$, $L10a_{165}$,\\
		&$L10a_{167}$, $L10a_{168}$, $L10a_{172}$, $L10a_{174}$, $L10n_{28}$, $L10n_{72}$, $L10n_{73}$, $L10n_{80}$,\\ 
		&$L10n_{87}$, $L10n_{88}$, $L10n_{89}$, $L10n_{99}$, $L10n_{100}$, $L10n_{101}$, $L10n_{102}$, $L10n_{103}$,\\ 
		&$L10n_{112}$, $L10n_{113}$, $(\ldots)$ \\
		\hline
		$6$&$L12a_{73}$, $L12a_{77}$, $L12a_{100}$, $L12a_{118}$, $L12a_{240}$, $L12a_{269}$, $L12a_{289}$, $L12a_{410}$,\\
		&$L12a_{413}$, $L12a_{414}$, $L12a_{447}$, $L12a_{523}$, $L12a_{527}$, $L12a_{643}$, $L12a_{646}$, $L12a_{684}$,\\
		&$L12a_{734}$, $L12a_{736}$, $L12a_{813}$, $L12a_{838}$, $L12a_{896}$, $L12a_{941}$, $L12a_{982}$, $L12a_{1012}$, \\
		&$L12a_{1029}$, $L12a_{1064}$, $L12a_{1259}$, $L12a_{1275}$, $L12a_{1339}$, $L12a_{1372}$, $L12a_{1400}$, \\
		&$L12a_{1413}$, $L12a_{1417}$, $L12a_{1434}$, $L12a_{1437}$, $L12a_{1443}$, $L12a_{1484}$, $L12a_{1491}$, \\
		&$L12a_{1501}$, $L12a_{1507}$, $L12a_{1532}$, $L12a_{1535}$, $L12a_{1574}$, $L12a_{1617}$, $L12a_{1635}$, \\
		&$L12a_{1658}$, $L12a_{1668}$, $L12a_{1676}$, $L12a_{1686}$, $L12a_{1732}$, $L12a_{1750}$, $L12a_{1762}$, \\
		&$L12a_{1787}$, $L12a_{1807}$, $L12a_{1847}$, $L12a_{1879}$, $L12a_{1881}$, $L12a_{1898}$, $L12a_{1910}$, \\
		&$L12a_{1913}$, $L12a_{1919}$, $L12a_{1921}$, $L12a_{1925}$, $L12a_{1930}$, $L12a_{1931}$, $L12a_{1945}$, \\
		&$L12a_{1971}$, $L12a_{1972}$, $L12a_{1984}$, $L12a_{1985}$, $L12a_{1996}$, $L12a_{2000}$, $L12a_{2007}$, \\
		&$L12a_{2009}$, $L12a_{2011}$, $L12a_{2012}$, $L12a_{2017}$, $L12a_{2019}$, $L12a_{2020}$, $(\ldots)$ \\
\end{longtable}
\end{center}

In \cite{ACFIPVWZ14} it is stated that the following knots have triple-crossing number equal to five: $7_3$, $7_5$, $8_4$, $8_6$, $8_8$, $8_{11}$, $8_{13}$, $8_{14}$, $8_{15}$, $9_2$, $9_5$, $9_8$, $9_{12}$, $9_{14}$, $9_{15}$, $9_{19}$, $9_{21}$, $9_{25}$, $9_{35}$, $9_{37}$, $9_{39}$, $9_{41}$. We verify this result by an output of our mentioned computational method, and we calculated the triple-crossing number of further prime knots, with the result presented in Table\;\ref{table3}. For knots with the triple-crossing number equal to six or seven we also used \cite[Corollary\;3.5. and 3.6.]{Ada13}.
\par
Minimal diagrams of (unoriented) prime knots with triple-crossing number up to four are presented (in the form of its $sPD$ codes) in Table\;\ref{table2}, the case for prime links is presented in Table\;\ref{table5}. An $sPD$ codes for the minimal knots with triple-crossing number equal to five can be found in the arXiv source file of this article's preprint version.

\begin{remark}
A set of triple-crossing names, for two minimal diagrams of the same knot on the same projection, may be very different. An example of such diagrams,for the knot $9_{46}$, is presented in Figure\;\ref{r26} where the names in one case are $\{eY, eY, eY, eY\}$ and in the other case are $\{eY, eY, eY, eX\}$.
\end{remark}

\begin{figure}[h!t]
    \centering
    \def\svgwidth{0.94\columnwidth}
    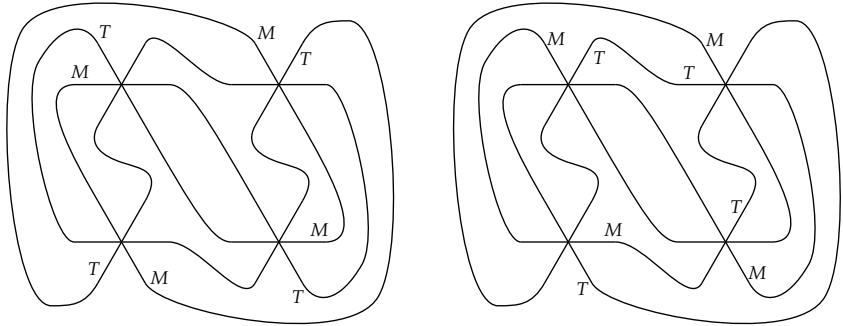
    \caption{Different minimal projections of the same knot.}
    \label{r26}
\end{figure}

It is obvious that $c_2(K)\leq 3c_3(K)$ by splitting open each triple crossing. It is interesting to have a knot $K$ such that $c_2(K)=3c_3(K)$. In \cite{Nis19} it is shown that there is a prime knot $K$ such that $c_2(K)=3c_3(K)=12$. Counting components to the dual graphs to the graphs shown in Figure\;\ref{r09} led us to the following.

\begin{proposition}
There is no prime knot $K$ such that $c_2(K)=3c_3(K)=15$.
\end{proposition}

\begin{question}
Is there a prime knot $K$ with an odd crossing number $c_2(K)=3c_3(K)$?
\end{question}

\begin{footnotesize}
\renewcommand{\arraystretch}{1.25}

\begin{center}
\begin{longtable}{r||l}
\caption{Minimal diagrams of prime knots.\label{table2}}\\
knot & $sPD$ code of a minimal knot diagram \\
\hline
\endfirsthead
\multicolumn{2}{c}
{\tablename\ \thetable\ -- \textit{Continued from previous page}} \\
knot & $sPD$ code of a minimal knot diagram \\
\hline
\endhead
\hline \multicolumn{2}{r}{\textit{Continued on next page}} \\
\endfoot
\hline
\endlastfoot
$3_1$&sPD[eY[4, 2, 5, 5, 1, 6], eY[3, 1, 2, 4, 6, 3]]\\
$4_1$&sPD[eX[5, 2, 4, 6, 1, 5], eY[3, 1, 2, 4, 6, 3]]\\
$5_2$&sPD[eY[6, 2, 7, 7, 1, 8], eY[3, 1, 2, 4, 9, 3], 
  eY[5, 9, 4, 6, 8, 5]]\\
$6_1$&sPD[eX[7, 2, 6, 8, 1, 7], eY[3, 1, 2, 4, 9, 3], 
  eY[5, 9, 4, 6, 8, 5]]\\
$5_1$&sPD[eY[8, 2, 7, 9, 1, 8], eY[5, 3, 4, 6, 2, 5], 
  eY[11, 4, 12, 12, 3, 1], eY[9, 7, 10, 10, 6, 11]]\\
$6_2$&sPD[eX[7, 2, 8, 8, 1, 9], eY[5, 3, 4, 6, 2, 5], 
  eY[11, 4, 12, 12, 3, 1], eY[9, 7, 10, 10, 6, 11]]\\
$6_3$&sPD[eX[7, 2, 8, 8, 1, 9], eX[10, 7, 9, 11, 6, 10], 
  eY[5, 3, 4, 6, 2, 5], eY[11, 4, 12, 12, 3, 1]]\\
$7_2$&sPD[eY[8, 2, 9, 9, 1, 10], eY[3, 1, 2, 4, 12, 3], 
  eY[5, 12, 4, 6, 11, 5], eY[7, 11, 6, 8, 10, 7]]\\
$7_4$&sPD[eX[7, 2, 8, 8, 1, 9], eX[5, 3, 6, 6, 2, 7], 
  eX[12, 4, 11, 1, 3, 12], eX[10, 5, 9, 11, 4, 10]]\\
$7_6$&sPD[eX[4, 3, 5, 5, 2, 6], eY[8, 2, 7, 9, 1, 8], 
  eY[11, 4, 12, 12, 3, 1], eY[9, 7, 10, 10, 6, 11]]\\
$7_7$&sPD[eX[7, 2, 8, 8, 1, 9], eX[12, 4, 11, 1, 3, 12], 
  eY[5, 3, 4, 6, 2, 5], eY[9, 7, 10, 10, 6, 11]]\\
$8_1$&sPD[eX[9, 2, 8, 10, 1, 9], eY[3, 1, 2, 4, 12, 3], 
  eY[5, 12, 4, 6, 11, 5], eY[7, 11, 6, 8, 10, 7]]\\
$8_3$&sPD[eX[7, 2, 8, 8, 1, 9], eX[5, 3, 6, 6, 2, 7], 
  eY[11, 4, 12, 12, 3, 1], eY[9, 5, 10, 10, 4, 11]]\\
$8_{12}$&sPD[eX[7, 2, 8, 8, 1, 9], eX[4, 3, 5, 5, 2, 6], 
  eY[11, 4, 12, 12, 3, 1], eY[9, 7, 10, 10, 6, 11]]\\
$8_{20}$&sPD[eX[9, 2, 10, 10, 1, 11], eY[8, 6, 11, 9, 5, 12], 
  eY[12, 5, 2, 1, 4, 3], eY[7, 4, 6, 8, 3, 7]]\\
$8_{21}$&sPD[eY[4, 3, 12, 5, 2, 1], eY[11, 9, 5, 12, 8, 6], 
  eY[10, 2, 9, 11, 1, 10], eY[7, 4, 6, 8, 3, 7]]\\
$9_{42}$&sPD[eX[9, 2, 10, 10, 1, 11], eX[6, 4, 7, 7, 3, 8], 
  eY[8, 6, 11, 9, 5, 12], eY[12, 5, 2, 1, 4, 3]]\\
$9_{44}$&sPD[eX[9, 2, 10, 10, 1, 11], eY[4, 3, 12, 5, 2, 1], 
  eY[11, 9, 5, 12, 8, 6], eY[7, 4, 6, 8, 3, 7]]\\
$9_{45}$&sPD[eY[8, 6, 11, 9, 5, 12], eY[12, 5, 2, 1, 4, 3], 
  eY[10, 2, 9, 11, 1, 10], eY[7, 4, 6, 8, 3, 7]]\\	
$9_{46}$&sPD[eX[7, 2, 6, 8, 1, 7], eY[12, 3, 9, 1, 2, 10], 
  eY[4, 9, 3, 5, 8, 4], eY[10, 6, 11, 11, 5, 12]]\\
$9_{48}$&sPD[eX[7, 2, 6, 8, 1, 7], eX[3, 9, 4, 4, 8, 5], 
  eX[11, 6, 10, 12, 5, 11], eY[9, 1, 2, 10, 12, 3]]\\
$10_{132}$&sPD[eY[6, 12, 3, 7, 11, 4], eY[2, 1, 5, 3, 12, 6], 
  eY[9, 2, 8, 10, 1, 9], eY[4, 11, 7, 5, 10, 8]]\\
$10_{136}$&sPD[eX[9, 2, 10, 10, 1, 11], eX[6, 4, 7, 7, 3, 8], 
  eY[4, 3, 12, 5, 2, 1], eY[11, 9, 5, 12, 8, 6]]\\
$10_{137}$&sPD[eX[9, 2, 10, 10, 1, 11], eY[7, 4, 6, 8, 3, 7], 
  eY[5, 12, 8, 6, 11, 9], eY[2, 1, 4, 3, 12, 5]]\\
$10_{140}$&sPD[eY[4, 10, 1, 5, 9, 2], eY[2, 12, 6, 3, 11, 7], 
  eY[12, 9, 5, 1, 8, 6], eY[7, 11, 3, 8, 10, 4]]\\
$K11n_{38}$&sPD[eX[8, 2, 9, 9, 1, 10], eY[6, 12, 3, 7, 11, 4], 
  eY[2, 1, 5, 3, 12, 6], eY[4, 11, 7, 5, 10, 8]]\\
$K11n_{139}$&sPD[eX[4, 2, 9, 5, 1, 10], eX[10, 8, 3, 11, 7, 4], 
  eX[2, 7, 11, 3, 6, 12], eX[8, 1, 5, 9, 12, 6]]\\
$K12n_{462}$&sPD[eX[4, 2, 9, 5, 1, 10], eY[10, 4, 7, 11, 3, 8], 
  eY[6, 3, 11, 7, 2, 12], eX[8, 1, 5, 9, 12, 6]]\\
	\hline
\end{longtable}
\end{center}
\end{footnotesize}


\begin{footnotesize}
\renewcommand{\arraystretch}{1.25}

\begin{center}
\begin{longtable}{r||l}
\caption{Minimal diagrams of prime links.\label{table5}}\\
link & $sPD$ code of a minimal link diagram \\
\hline
\endfirsthead
\multicolumn{2}{c}
{\tablename\ \thetable\ -- \textit{Continued from previous page}} \\
link & $sPD$ code of a minimal link diagram \\
\hline
\endhead
\hline \multicolumn{2}{r}{\textit{Continued on next page}} \\
\endfoot
\hline
\endlastfoot

$2^2_1$&sPD[eX[1, 1, 3, 2, 2, 3]]\\
$4^2_1$&sPD[eY[4, 3, 5, 5, 2, 6], eY[3, 4, 1, 1, 6, 2]]\\
$6^3_3$&sPD[eX[1, 4, 6, 2, 3, 5], eX[5, 3, 2, 6, 4, 1]]\\
$5^2_1$&sPD[eX[8, 1, 7, 9, 2, 8], eY[3, 2, 4, 4, 1, 5], 
  eY[6, 3, 5, 7, 9, 6]]\\
$6^2_1$&sPD[eX[7, 1, 8, 8, 2, 9], eX[5, 3, 6, 6, 1, 7], 
  eX[9, 2, 4, 4, 3, 5]]\\
$6^2_3$&sPD[eX[5, 3, 6, 6, 9, 7], eY[7, 1, 8, 8, 2, 9], 
  eY[3, 2, 4, 4, 1, 5]]\\
$6^3_1$&sPD[eY[7, 1, 8, 8, 2, 9], eY[3, 2, 4, 4, 1, 5], 
  eY[5, 7, 6, 6, 9, 3]]\\
$7^2_7$&sPD[eX[3, 8, 1, 4, 9, 2], eX[7, 9, 4, 1, 8, 5], 
  eY[6, 3, 5, 7, 2, 6]]\\
$7^2_8$&sPD[eY[8, 3, 2, 9, 4, 1], eY[6, 3, 5, 7, 2, 6], 
  eY[7, 5, 8, 1, 4, 9]]\\
$8^2_{15}$&sPD[eX[5, 3, 6, 6, 2, 7], eY[8, 3, 2, 9, 4, 1], 
  eY[7, 5, 8, 1, 4, 9]]\\
$8^2_{16}$&sPD[eX[5, 3, 6, 6, 2, 7], eX[3, 8, 1, 4, 9, 2], 
  eX[7, 9, 4, 1, 8, 5]]\\
$8^3_7$&sPD[eX[6, 2, 4, 7, 3, 5], eX[8, 1, 9, 9, 2, 6], 
  eX[5, 3, 7, 4, 1, 8]]\\
$8^3_8$&sPD[eX[6, 2, 9, 7, 3, 6], eY[8, 4, 2, 9, 5, 1], 
  eY[3, 4, 8, 1, 5, 7]]\\
$9^2_{49}$&sPD[eX[2, 7, 5, 3, 8, 6], eX[6, 8, 3, 1, 9, 4], 
  eX[4, 9, 1, 5, 7, 2]]\\
$6^2_2$&sPD[eX[6, 1, 7, 7, 2, 8], eX[4, 3, 5, 5, 1, 6], 
  eX[10, 2, 11, 11, 3, 12], eX[9, 4, 8, 10, 12, 9]]\\
$6^3_2$&sPD[eX[8, 2, 7, 9, 1, 8], eX[5, 11, 6, 6, 12, 7], 
  eY[3, 1, 4, 4, 2, 5], eY[10, 12, 9, 3, 11, 10]]\\
$7^2_1$&sPD[eX[10, 2, 11, 11, 3, 12], eX[9, 4, 8, 10, 12, 9], 
  eY[7, 1, 6, 8, 2, 7], eY[5, 3, 4, 6, 1, 5]]\\
$7^2_2$&sPD[eX[5, 3, 4, 6, 2, 5], eY[8, 2, 9, 9, 1, 10], 
  eY[10, 1, 11, 11, 3, 12], eY[7, 4, 6, 8, 12, 7]]\\
$7^2_3$&sPD[eX[6, 1, 5, 7, 2, 6], eX[7, 5, 8, 8, 4, 9], 
  eX[12, 10, 3, 3, 9, 4], eY[10, 2, 11, 11, 1, 12]]\\
$7^2_5$&sPD[eY[4, 8, 3, 5, 7, 4], eY[8, 7, 1, 1, 6, 2], 
  eY[9, 6, 10, 10, 5, 11], eY[11, 3, 12, 12, 2, 9]]\\
$7^3_1$&sPD[eX[1, 5, 2, 2, 4, 3], eX[10, 3, 11, 11, 4, 12], 
  eY[6, 1, 7, 7, 5, 8], eY[8, 12, 9, 9, 10, 6]]\\
$8^2_1$&sPD[eX[7, 1, 8, 8, 2, 9], eX[5, 4, 6, 6, 1, 7], 
  eX[11, 3, 12, 12, 4, 5], eX[9, 2, 10, 10, 3, 11]]\\
$8^2_3$&sPD[eX[6, 1, 7, 7, 2, 8], eX[4, 3, 5, 5, 1, 6], 
  eX[10, 2, 11, 11, 3, 12], eY[8, 4, 9, 9, 12, 10]]\\
$8^2_6$&sPD[eX[7, 5, 8, 8, 4, 9], eX[12, 10, 3, 3, 9, 4], 
  eY[5, 1, 6, 6, 2, 7], eY[10, 2, 11, 11, 1, 12]]\\
$8^2_9$&sPD[eX[3, 8, 4, 4, 7, 5], eY[8, 7, 1, 1, 6, 2], 
  eY[9, 6, 10, 10, 5, 11], eY[11, 3, 12, 12, 2, 9]]\\
$8^3_1$&sPD[eY[6, 2, 5, 7, 1, 6], eY[10, 2, 11, 11, 1, 12], 
  eY[4, 9, 3, 5, 10, 4], eY[8, 12, 7, 3, 9, 8]]\\	
$8^3_3$&sPD[eX[4, 2, 3, 5, 1, 4], eY[11, 3, 10, 12, 2, 11], 
  eY[8, 1, 7, 9, 5, 8], eY[6, 10, 9, 7, 12, 6]]\\
$8^3_4$&sPD[eX[2, 6, 3, 3, 7, 4], eY[5, 11, 2, 6, 12, 1], 
  eX[10, 1, 5, 11, 4, 8], eY[12, 7, 9, 9, 8, 10]]\\
$8^3_9$&sPD[eX[4, 1, 5, 5, 2, 6], eX[8, 2, 9, 9, 1, 10], 
  eX[7, 12, 6, 8, 11, 7], eY[10, 11, 3, 3, 12, 4]]\\
$8^3_{10}$&sPD[eX[7, 4, 6, 8, 5, 7], eY[3, 12, 5, 4, 11, 1], 
  eY[1, 11, 6, 2, 10, 9], eX[12, 8, 2, 10, 9, 3]]\\
$8^4_1$&sPD[eY[8, 1, 9, 9, 2, 10], eY[3, 2, 4, 4, 1, 5], 
  eY[5, 12, 6, 6, 11, 3], eY[10, 11, 7, 7, 12, 8]]\\
$8^4_2$&sPD[eX[7, 2, 8, 8, 1, 9], eX[4, 2, 3, 5, 1, 4], 
  eX[6, 12, 5, 3, 11, 6], eY[10, 12, 9, 7, 11, 10]]\\
$8^4_3$&sPD[eX[7, 2, 8, 8, 1, 9], eX[4, 2, 3, 5, 1, 4], 
  eY[5, 12, 6, 6, 11, 3], eY[10, 12, 9, 7, 11, 10]]\\
$9^2_{45}$&sPD[eX[2, 12, 10, 3, 11, 1], eX[7, 11, 4, 8, 12, 5], 
  eY[6, 2, 5, 7, 1, 6], eY[9, 4, 8, 10, 3, 9]]\\
$9^2_{46}$&sPD[eY[11, 4, 2, 12, 5, 1], eY[7, 4, 6, 8, 3, 7], 
  eY[9, 3, 8, 10, 2, 9], eY[10, 6, 11, 1, 5, 12]]\\
$9^2_{47}$&sPD[eY[2, 1, 11, 3, 10, 12], eY[6, 2, 5, 7, 1, 6], 
  eY[9, 4, 8, 10, 3, 9], eY[12, 8, 4, 11, 7, 5]]\\
$9^2_{48}$&sPD[eX[3, 11, 9, 4, 12, 10], eX[10, 12, 2, 1, 11, 3], 
  eY[5, 2, 6, 6, 1, 7], eY[8, 5, 7, 9, 4, 8]]\\
$9^2_{50}$&sPD[eX[5, 2, 4, 6, 1, 5], eX[2, 1, 10, 3, 8, 9], 
  eX[12, 7, 11, 9, 8, 12], eX[10, 6, 4, 11, 7, 3]]\\
$9^2_{52}$&sPD[eX[12, 7, 11, 9, 8, 12], eY[8, 3, 10, 1, 2, 9], 
  eY[4, 2, 5, 5, 1, 6], eY[10, 3, 7, 11, 4, 6]]\\
$9^2_{53}$&sPD[eX[10, 6, 9, 11, 7, 10], eX[4, 11, 3, 5, 12, 4], 
  eX[2, 8, 1, 3, 9, 2], eX[7, 12, 5, 1, 8, 6]]\\
$9^2_{54}$&sPD[eX[10, 6, 9, 11, 7, 10], eX[4, 11, 3, 5, 12, 4], 
  eY[1, 8, 2, 2, 9, 3], eY[12, 7, 6, 8, 1, 5]]\\
$9^3_{13}$&sPD[eX[5, 1, 7, 6, 2, 8], eX[12, 4, 6, 7, 1, 5], 
  eY[10, 9, 11, 11, 8, 12], eY[2, 9, 3, 3, 10, 4]]\\
$9^3_{14}$&sPD[eX[3, 10, 2, 4, 9, 3], eY[5, 11, 2, 6, 12, 1], 
  eY[8, 11, 7, 9, 10, 8], eY[4, 7, 5, 1, 12, 6]]\\	
$9^3_{18}$&sPD[eX[7, 2, 8, 8, 3, 9], eY[1, 4, 12, 2, 3, 11], 
  eY[10, 5, 9, 7, 6, 10], eY[11, 5, 6, 12, 4, 1]]\\
$9^3_{19}$&sPD[eY[3, 12, 5, 4, 11, 1], eY[1, 11, 6, 2, 10, 9], 
  eX[12, 8, 2, 10, 9, 3], eY[6, 4, 7, 7, 5, 8]]\\
$L10n_7$&sPD[eX[6, 4, 7, 7, 3, 8], eX[8, 3, 9, 9, 2, 10], 
  eX[4, 11, 1, 5, 12, 2], eX[10, 12, 5, 1, 11, 6]]\\
$L10n_8$&sPD[eX[6, 4, 7, 7, 3, 8], eX[8, 3, 9, 9, 2, 10], 
  eY[11, 4, 2, 12, 5, 1], eY[10, 6, 11, 1, 5, 12]]\\
$L10n_{9}$&sPD[eX[10, 3, 11, 1, 2, 12], eX[8, 4, 9, 9, 3, 10], 
  eX[12, 5, 7, 11, 4, 8], eY[6, 2, 5, 7, 1, 6]]\\
$L10n_{19}$&sPD[eX[6, 2, 5, 7, 1, 6], eX[3, 11, 9, 4, 12, 10], 
  eX[10, 12, 2, 1, 11, 3], eY[8, 5, 7, 9, 4, 8]]\\
$L10n_{31}$&sPD[eX[5, 2, 6, 6, 1, 7], eX[2, 12, 10, 3, 11, 1], 
  eX[7, 11, 4, 8, 12, 5], eY[9, 4, 8, 10, 3, 9]]\\
$L10n_{45}$&sPD[eX[2, 1, 10, 3, 8, 9], eX[12, 7, 11, 9, 8, 12], 
  eX[10, 6, 4, 11, 7, 3], eY[4, 2, 5, 5, 1, 6]]\\
$L10n_{48}$&sPD[eX[5, 2, 4, 6, 1, 5], eX[12, 7, 11, 9, 8, 12], 
  eY[8, 3, 10, 1, 2, 9], eY[10, 3, 7, 11, 4, 6]]\\
$L10n_{49}$&sPD[eX[5, 2, 4, 6, 1, 5], eX[12, 7, 11, 9, 8, 12], 
  eX[8, 9, 2, 1, 10, 3], eX[3, 10, 6, 4, 11, 7]]\\
$L10n_{68}$&sPD[eX[5, 1, 7, 6, 2, 8], eX[3, 9, 2, 4, 10, 3], 
  eX[12, 4, 6, 7, 1, 5], eY[10, 9, 11, 11, 8, 12]]\\
$L10n_{69}$&sPD[eX[3, 9, 2, 4, 10, 3], eY[1, 5, 8, 2, 6, 7], 
  eY[12, 5, 1, 7, 6, 4], eY[10, 9, 11, 11, 8, 12]]\\
$L10n_{70}$&sPD[eX[7, 2, 8, 8, 3, 9], eX[9, 5, 10, 10, 6, 7], 
  eY[1, 4, 12, 2, 3, 11], eY[11, 5, 6, 12, 4, 1]]\\
$L10n_{71}$&sPD[eX[7, 2, 8, 8, 1, 9], eX[11, 4, 10, 12, 3, 11], 
  eX[2, 5, 4, 3, 6, 1], eX[9, 6, 12, 10, 5, 7]]\\
$L10n_{77}$&sPD[eX[6, 1, 8, 7, 2, 9], eX[12, 5, 11, 8, 1, 12], 
  eX[3, 9, 2, 4, 10, 3], eX[10, 4, 7, 11, 5, 6]]\\
$L10n_{78}$&sPD[eX[7, 2, 12, 8, 3, 7], eX[9, 3, 8, 10, 4, 9], 
  eY[11, 5, 2, 12, 6, 1], eY[4, 5, 11, 1, 6, 10]]\\
$L10n_{81}$&sPD[eX[12, 5, 11, 8, 1, 12], eX[3, 9, 2, 4, 10, 3], 
  eY[1, 6, 9, 2, 7, 8], eY[10, 6, 5, 11, 7, 4]]\\
$L10n_{92}$&sPD[eX[2, 6, 3, 3, 7, 4], eY[5, 11, 2, 6, 12, 1], 
  eY[12, 7, 9, 9, 8, 10], eY[4, 11, 5, 1, 10, 8]]\\	
$L10n_{93}$&sPD[eX[9, 7, 12, 10, 8, 9], eX[2, 6, 3, 3, 7, 4], 
  eX[11, 5, 1, 12, 6, 2], eX[4, 8, 10, 1, 5, 11]]\\
$L10n_{94}$&sPD[eX[7, 2, 6, 8, 1, 7], eY[2, 11, 5, 3, 12, 1], 
  eX[12, 8, 4, 10, 9, 3], eX[4, 6, 11, 5, 9, 10]]\\
$L10n_{104}$&sPD[eX[5, 1, 6, 6, 2, 7], eY[2, 10, 9, 3, 11, 8], 
  eY[9, 12, 1, 8, 11, 3], eY[4, 10, 7, 5, 12, 4]]\\
$L10n_{105}$&sPD[eX[10, 2, 8, 11, 1, 9], eX[9, 1, 11, 8, 3, 12], 
  eY[6, 3, 5, 7, 2, 6], eY[4, 10, 7, 5, 12, 4]]\\
$L10n_{108}$&sPD[eX[10, 2, 8, 11, 3, 9], eX[5, 2, 4, 6, 1, 5], 
  eX[7, 12, 6, 4, 10, 7], eX[9, 3, 11, 8, 1, 12]]\\
$L10n_{109}$&sPD[eX[5, 2, 4, 6, 3, 5], eX[7, 10, 6, 4, 12, 7], 
  eY[11, 8, 2, 12, 9, 1], eY[3, 8, 11, 1, 9, 10]]\\
$L11n_{140}$&sPD[eY[9, 6, 2, 10, 5, 1], eX[7, 2, 10, 8, 3, 11], 
  eY[11, 7, 12, 12, 6, 4], eY[3, 4, 9, 1, 5, 8]]\\	
$L11n_{141}$&sPD[eX[9, 1, 5, 10, 2, 6], eX[7, 2, 10, 8, 3, 11], 
  eX[4, 3, 8, 5, 1, 9], eY[11, 7, 12, 12, 6, 4]]\\
$L11n_{204}$&sPD[eX[4, 11, 5, 5, 12, 6], eX[2, 9, 7, 3, 10, 8], 
  eX[8, 10, 3, 1, 11, 4], eX[6, 12, 1, 7, 9, 2]]\\
$L11n_{376}$&sPD[eY[11, 4, 2, 12, 5, 1], eX[3, 9, 6, 4, 10, 7], 
  eX[7, 10, 2, 8, 9, 3], eY[8, 6, 11, 1, 5, 12]]\\
$L11n_{378}$&sPD[eY[11, 4, 2, 12, 5, 1], eY[3, 7, 10, 4, 6, 9], 
  eY[8, 6, 11, 1, 5, 12], eY[9, 8, 2, 10, 7, 3]]\\	
$L11n_{419}$&sPD[eY[2, 11, 5, 3, 12, 1], eX[12, 8, 4, 10, 9, 3], 
  eX[4, 6, 11, 5, 9, 10], eY[6, 2, 7, 7, 1, 8]]\\
$L11n_{420}$&sPD[eX[7, 2, 6, 8, 1, 7], eX[9, 10, 4, 6, 11, 5], 
  eX[5, 11, 2, 1, 12, 3], eX[3, 12, 8, 4, 10, 9]]\\
$L12n_{1804}$&sPD[eY[1, 11, 4, 2, 12, 5], eY[5, 12, 8, 6, 11, 1], 
  eX[3, 9, 6, 4, 10, 7], eX[7, 10, 2, 8, 9, 3]]\\
$L12n_{1806}$&sPD[eX[11, 1, 5, 12, 2, 4], eX[3, 9, 6, 4, 10, 7], 
  eX[5, 1, 11, 6, 8, 12], eX[7, 10, 2, 8, 9, 3]]\\
$L12n_{1807}$&sPD[eX[11, 1, 5, 12, 2, 4], eY[3, 7, 10, 4, 6, 9], 
  eX[5, 1, 11, 6, 8, 12], eY[9, 8, 2, 10, 7, 3]]\\
$L12n_{1997}$&sPD[eY[11, 2, 6, 12, 1, 5], eY[9, 4, 8, 10, 3, 7], 
  eY[8, 4, 11, 5, 1, 10], eY[6, 2, 9, 7, 3, 12]]\\
$L12n_{1998}$&sPD[eY[1, 5, 11, 2, 6, 12], eY[3, 7, 9, 4, 8, 10], 
  eY[12, 6, 2, 9, 7, 3], eY[10, 8, 4, 11, 5, 1]]\\
$L12n_{2150}$&sPD[eX[9, 2, 7, 10, 1, 8], eX[8, 1, 10, 7, 4, 11], 
  eX[11, 4, 6, 12, 3, 5], eX[6, 2, 9, 5, 3, 12]]\\
$L12n_{2151}$&sPD[eX[11, 1, 8, 12, 2, 7], eX[8, 1, 11, 7, 4, 10], 
  eY[9, 6, 4, 10, 5, 3], eY[2, 6, 9, 3, 5, 12]]\\
$L12n_{2159}$&sPD[eY[11, 7, 2, 12, 8, 1], eY[4, 7, 11, 1, 8, 10], 
  eY[9, 6, 4, 10, 5, 3], eY[2, 6, 9, 3, 5, 12]]\\
$L12n_{2206}$&sPD[eY[5, 7, 2, 6, 8, 1], eY[9, 4, 11, 7, 5, 12], 
  eY[12, 1, 8, 10, 3, 9], eY[2, 11, 4, 3, 10, 6]]\\
$L12n_{2209}$&sPD[eX[4, 2, 9, 5, 1, 8], eX[12, 8, 1, 10, 7, 3], 
  eY[9, 5, 10, 7, 6, 11], eY[2, 11, 6, 3, 12, 4]]\\
	\hline
\end{longtable}
\end{center}
\end{footnotesize}

\begin{figure}[h!t]
    \centering
		    \caption{Spherical graphs with five faces being hexagons.}
    \def\svgwidth{0.9\columnwidth}
    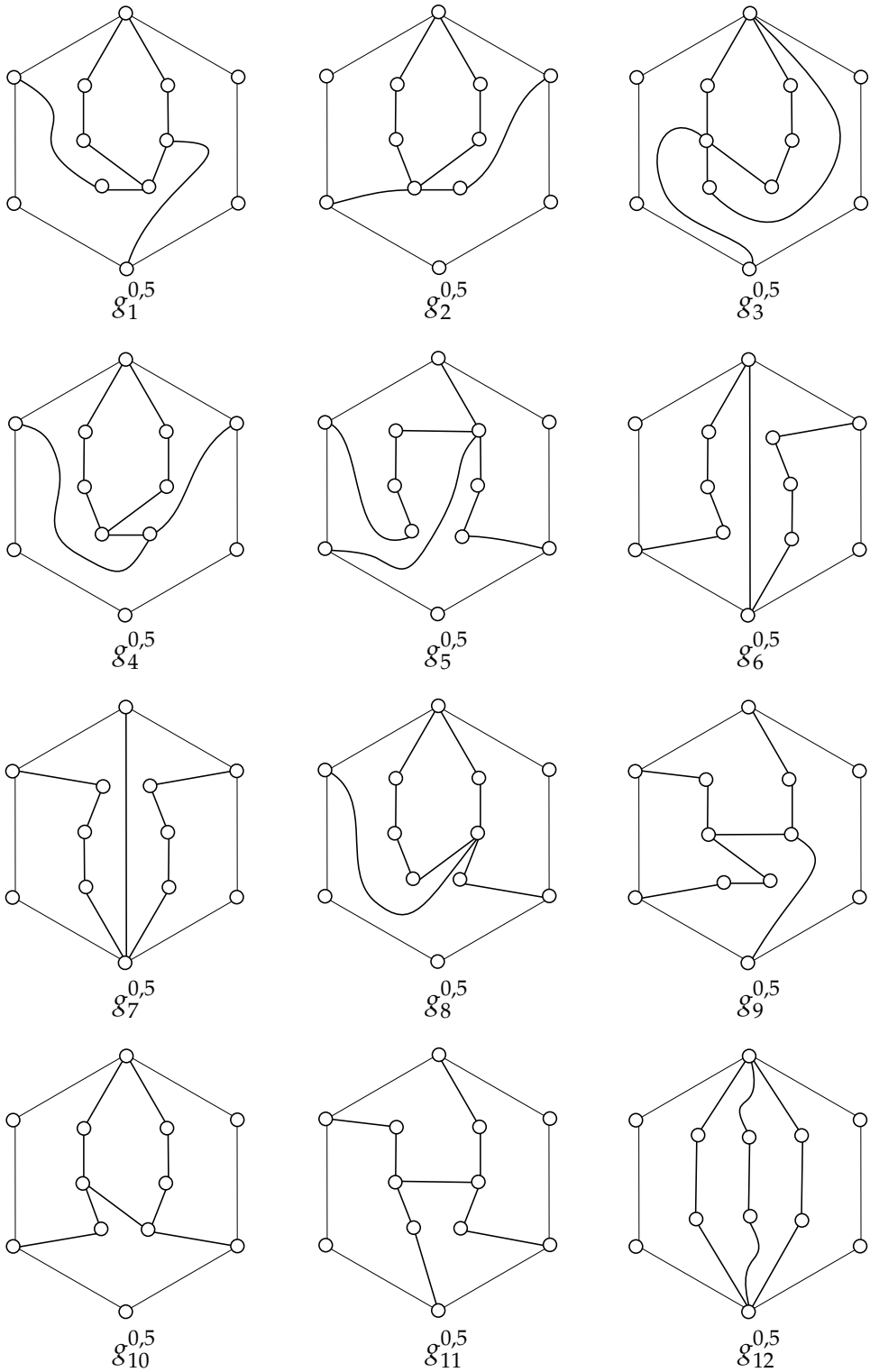
    \label{r09}
\end{figure}

\begin{figure}[h!t]
    \centering
		\caption{Table of projections with four triple-points.}
		\ \\
    \def\svgwidth{.99\columnwidth}
    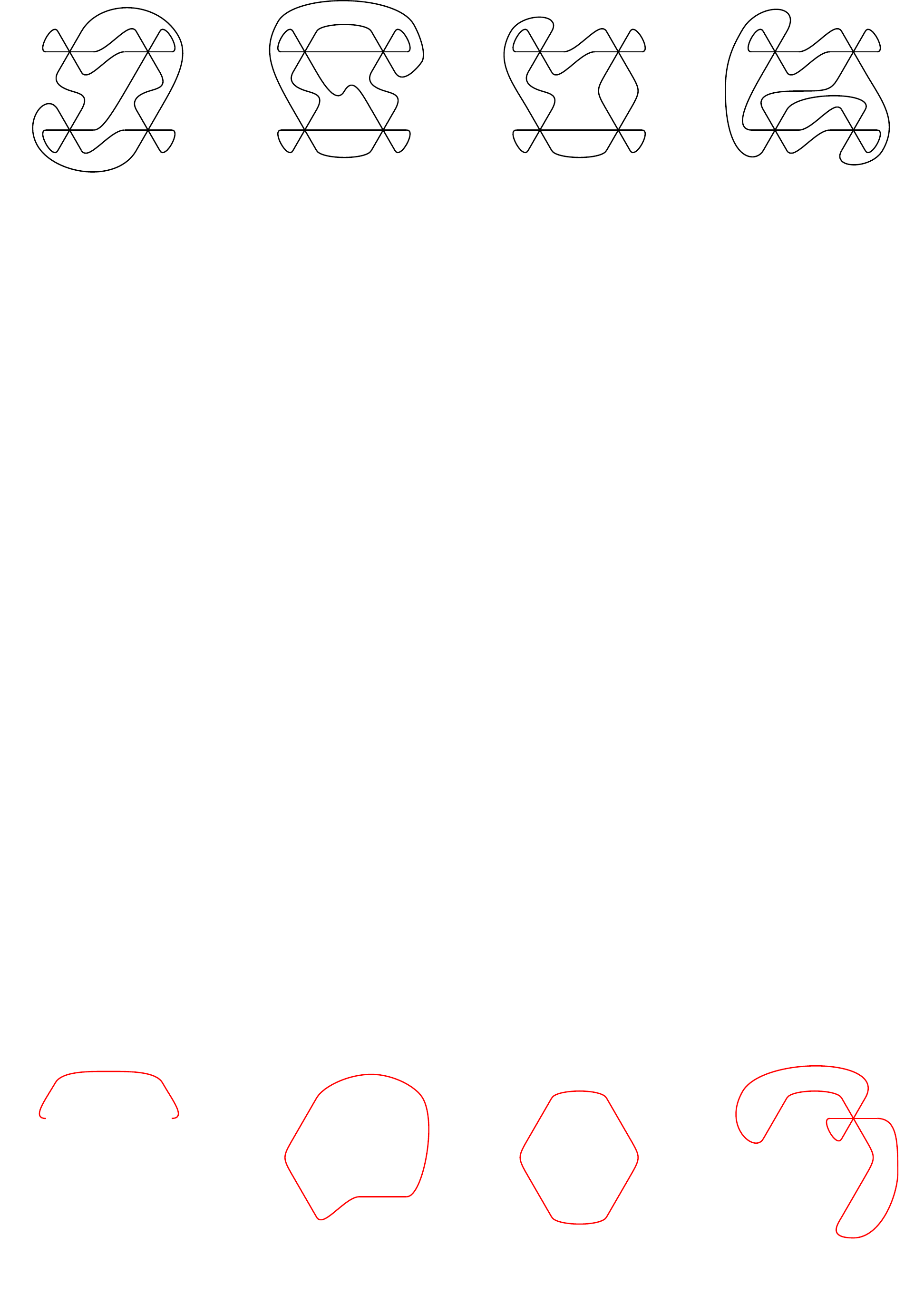
    \label{r15}
\end{figure}

\begin{figure}[h!t]
    \centering
		\caption{Table of projections with four triple-points. (cont.)}
		\ \\
    \def\svgwidth{.99\columnwidth}
    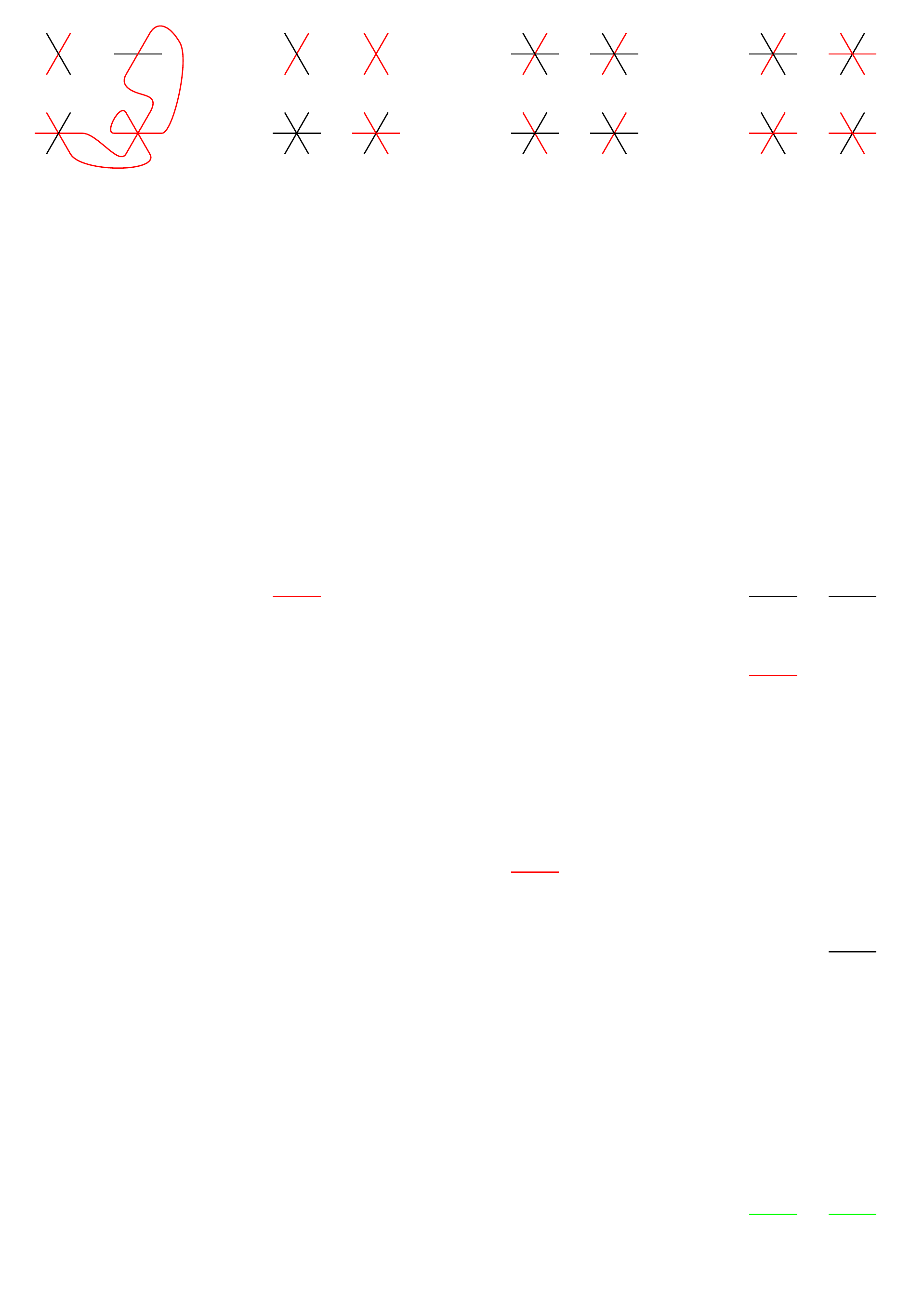
    \label{r16}
\end{figure}

\begin{figure}[h!t]
    \centering
		\caption{Table of projections with four triple-points. (cont.)}
		\ \\
    \def\svgwidth{.99\columnwidth}
    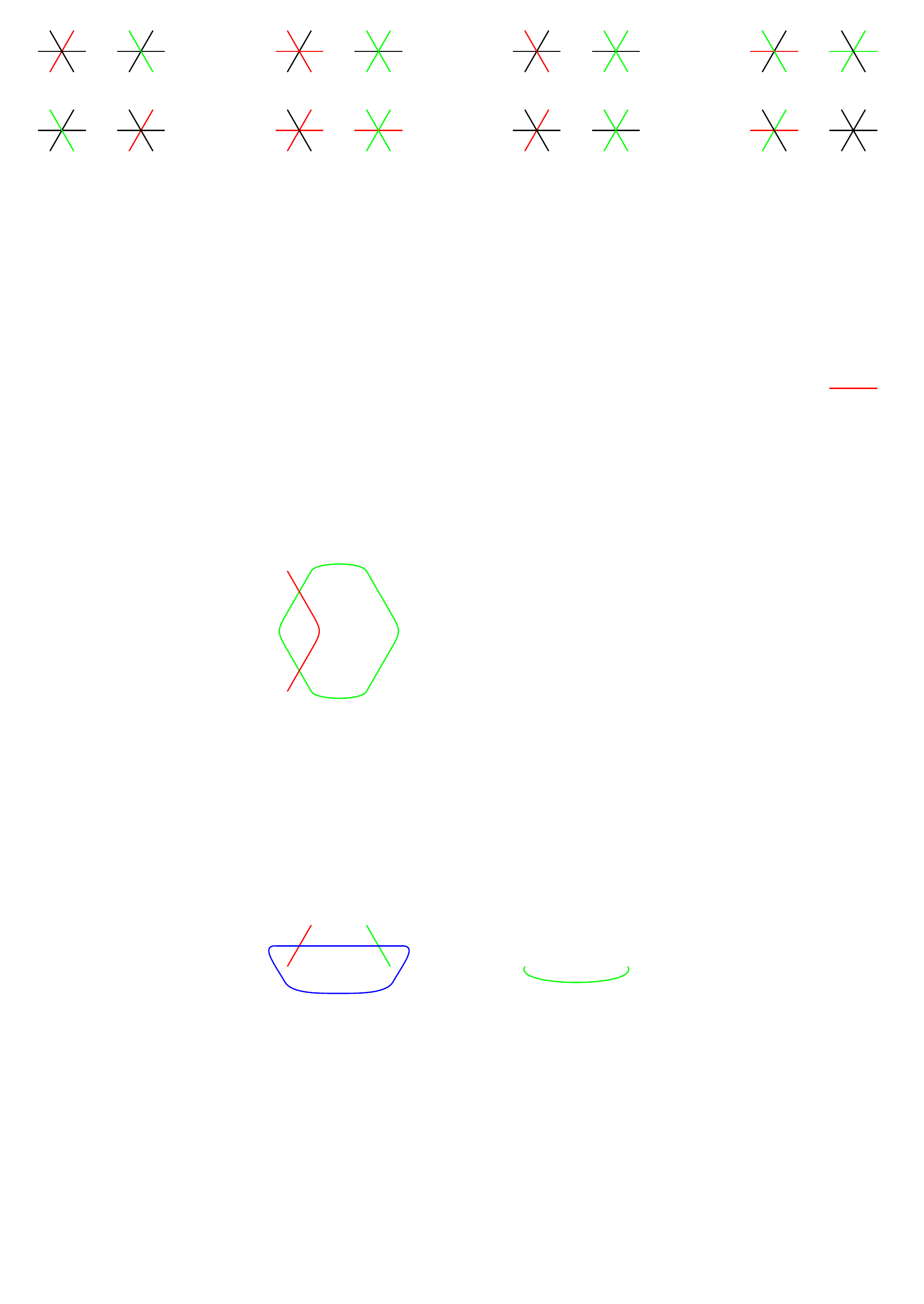
    \label{r13}
\end{figure}

\end{document}